\documentclass[12pt]{article}

\usepackage{amssymb}
\usepackage{amsfonts}
\usepackage{latexsym}
\usepackage{graphicx}
\usepackage{epstopdf}
\usepackage{amsmath,amsthm}
\usepackage{todonotes}

\newcommand{\ignore}[1]{}

\newcommand{\beqn}{\begin{eqnarray*}}
\newcommand{\eeqn}{\end{eqnarray*}}

\newcommand{\ii}{\mathrm{i}}
\newcommand{\ee}{\mathrm{e}}

\newcommand{\R}{\mathbb{R}}
\newcommand{\C}{\mathbb{C}}
\newcommand{\N}{\mathbb{N}}

\newcommand{\EE}{{\mathcal E}}
\newcommand{\GG}{{\mathcal G}}
\newcommand{\HH}{{\mathcal H}}
\newcommand{\KK}{{\mathcal K}}

\newcommand{\UU}{{\mathcal U}}
\newcommand{\VV}{{\mathcal V}}

\newcommand{\XX}{{\mathcal X}}

\newcommand{\Matrix}[1]{\ensuremath{\left[\begin{array}{ccccccccccccccccr} #1 \end{array}\right]}}
\newcommand{\Matrixc}[1]{\ensuremath{\left(\begin{array}{cccccccccccc} #1 \end{array}\right)}}
\newcommand{\vc}[1]{\mathbf{#1}}

\newtheorem{theorem}{Theorem}[section]
\newtheorem{lemma}[theorem]{Lemma}
\newtheorem{definition}[theorem]{Definition}
\newtheorem{corollary}[theorem]{Corollary}
\newtheorem{proposition}[theorem]{Proposition}
\newtheorem{remark}[theorem]{Remark}

\newtheorem{example}[theorem]{Example}

\renewcommand{\epsilon}{\varepsilon}
{\begingroup

  \begin{enumerate}}
  {\end{enumerate}\endgroup}
  
\usepackage{color}
\newcommand{\RED}[1]{{\color{red}{#1}}}

\title{Branching Ratios of Input Trees\\for Directed Multigraphs\thanks{This
is the authors' accepted manuscript of an article to appear in the
\emph{Journal of Algebraic Combinatorics}; it does not reflect post-acceptance
improvements or any corrections made by the publisher. The Version of Record is
available at \texttt{https://doi.org/10.1007/s10801-026-01618-4}.}}

\author{
Paolo Boldi \\ Computer Science Dept. \\ Universit\`a degli Studi di Milano \\ Milano, Italy \\
\texttt{paolo.boldi@unimi.it}
\and
Ian Stewart \\ Mathematics Institute
\\ University of Warwick \\ Coventry CV4 7AL, UK \\
\texttt{i.n.stewart@warwick.ac.uk}
}

\date{}

\begin{document}

\maketitle

\begin{abstract}
We define the branching ratio of the input tree of a node in a finite directed multigraph,
prove that it exists for every node,
and show that it is equal to the largest eigenvalue of the adjacency matrix of the induced subgraph
determined by all upstream nodes. This real eigenvalue
exists by the Perron--Frobenius Theorem for non-negative matrices. We motivate our analysis with simple
examples, obtain information about the asymptotics for the limit growth of the input tree, and
establish other basic properties of the branching ratio.
\end{abstract}


\section{Introduction}

By a `network' we mean a finite directed multigraph (that is, self-loops and multiple
edges are allowed).
The branching ratio (or fractal dimension) of a node $i$ in a network quantifies
the growth rate of its input tree; equivalently, of the number of walks of given length $\ell$
that terminate at node $i$. In simple examples, this number is asymptotic to
$C\rho^\ell$ for $C, \rho > 0$, and the branching ratio is $\rho$. Moreover,
$\rho$ is the spectral radius (maximal real eigenvalue, or Perron eigenvalue) of the
adjacency matrix of the subnetwork induced by all nodes that lie on
directed paths terminating at $i$.
We show that a similar result holds in general, except that the asymptotic form
of the growth rate can be more complicated. The result is not a surprise,
but a proof is more complicated than might be expected and we have found 
neither statements nor proofs for this fact in the literature, despite the extensive attention paid
to the spectral theory of graphs and digraphs---for example, in \cite{AO,B74,B10,BH90,B08,LT85,T84}. 

We encountered this problem in connection with
synchrony patterns in biological networks \cite{MBSS25}.
Because  the branching ratio has important biological applications,
we provide the required statements and proofs. 
The obvious approach using the Perron--Frobenius Theorem
runs into technical difficulties since the dominant terms in the asymptotics 
(see the proof of Lemma \ref{L:upper_bound}) might, in principle, cancel out. 
We use an indirect method to show that this does not happen.

\subsection{Motivating Examples}
\label{S:ME}

To motivate the  mathematics we give some examples. The first one considers a small network (`circuit') that occurs
in several kinds of biological systems, such as gene regulatory networks
and connectomes. Throughout the paper we employ the usual notations $\sim, O(\cdot), o(\cdot)$ for
asymptotic behaviour. Any terms not yet defined are defined later.

\begin{example}\em
\label{ex:fibo}

The {\em Fibonacci circuit} of \cite{LMRASM20, MM19} is shown in Fig.\ref{F:fibo}.
It is strongly connected, and has two nodes $1$, $2$ and three edges $e_1,e_2,e_3$.

\begin{figure}[h!]
\centerline{%
\includegraphics[width=0.2\textwidth]{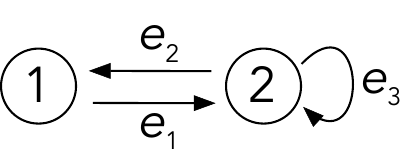}}
\caption{Fibonacci circuit.}
\label{F:fibo}
\end{figure}

\noindent
Its adjacency matrix is
\[
A = \Matrix{0 & 1\\1 & 1},
\]
and the eigenvalues of $A$ are $\frac{1\pm\sqrt{5}}{2}$, so the spectral radius (maximal or Perron eigenvalue, which we denote by $\xi$) is
$\xi=\frac{1+\sqrt{5}}{2} = \phi$, the golden ratio.

We focus on walks through this network,
which are sequences of edges joined head-to-tail; the length of a walk
is the number of edges in the sequence.
We ask: How many walks of length $\ell$ are there that terminate
at a given node $i\in \{1,2\}$? Call this number $a_i(\ell)$.
From the graph, tracking back one step from each node, we see that
\[
a_1(\ell+1) = a_2(\ell) \qquad \quad a_2(\ell+1) = a_1(\ell)+a_2(\ell)
\]
so
\begin{equation}
\label{E:fibo_rec}
\begin{array}{l}
a_1(\ell+2) = a_2(\ell+1) = a_1(\ell)+a_2(\ell) = a_1(\ell)+a_1(\ell+1) \\
a_2(\ell+2) = a_1(\ell+1)+a_2(\ell+1) = a_2(\ell)+a_2(\ell+1)
\end{array}
\end{equation}
Initial values are $a_1(0)=1, a_1(1)=1, a_2(0)=1, a_2(1)=2$.
Let $F_\ell$ be the $\ell$-th
Fibonacci number, that is, $F_0=F_1 = 1$ and $F_\ell = F_{\ell-1}+F_{\ell-2}$ for $\ell \geq 2$.
It is easy to see that the recurrences \eqref{E:fibo_rec} lead to:
\[
a_1(\ell) = F_\ell \qquad a_2(\ell)  = F_{\ell+1}.
\]
Therefore
\[
\frac{a_1(\ell+1)}{a_1(\ell)} = \frac{F_{\ell+1}}{F_\ell} \qquad \frac{a_2(\ell+1)}{a_2(\ell)} = \frac{F_{\ell+2}}{F_{\ell+1}}
\]
both of which tend to $\phi$ as $\ell \to \infty$. These two limits represent the asymptotic growth rate of the number of walks terminating at $i$ as $\ell \to \infty$, that is,
\begin{equation}
	\label{eqn:bad_branching_ratio_pre}
	\rho(i) = \lim_{\ell\to\infty} \frac{a_i(\ell+1)}{a_i(\ell)}
\end{equation}
exists for both nodes $i=1,2$ and equals $\phi$. 
Observe, moreover, the asymptotic expressions
\begin{equation}
\label{E:fibo_asymp}
a_1(\ell) \sim  \frac{1}{\sqrt{5}}\phi^{\ell+1} \qquad  a_2(\ell) \sim  \frac{1}{\sqrt{5}}\phi^{\ell+2}= \frac{\phi}{\sqrt{5}}\phi^{\ell+1}
\end{equation}
as $\ell \to \infty$.
\end{example}

\medskip
This example exhibits three distinct ways to think about the branching ratio:
as an eigenvalue, as a topological feature, and as an asymptotic growth rate.
The same goes for an arbitrary network.
All three interpretations are closely linked, and we repeatedly pass between them
in proofs. We do this because some properties are obvious in one representation
but not in others.

For this example the obvious way to define the branching ratio $\rho(i)$
of node $i$ (which motivates its name) is \eqref{eqn:bad_branching_ratio_pre}.
An alternative would be
\begin{equation}
\label{eqn:good_branching_ratio_pre}
\delta(i) = \lim_{\ell\to\infty} (a_i(\ell))^{1/\ell}.
\end{equation}
In the Fibonacci network both limits exist and give the same result.
However, we show below that the limit in \eqref{eqn:bad_branching_ratio_pre} need not exist,
whereas the sequence of \eqref{eqn:good_branching_ratio_pre} always converges. We therefore define the branching ratio
$\delta(i)$ for node $i$ by \eqref{eqn:good_branching_ratio_pre}. This choice loses no information:
whenever the limit in \eqref{eqn:bad_branching_ratio_pre} exists, it equals $\delta(i)$
(Remark \ref{R:ratio_vs_root}), so \eqref{eqn:good_branching_ratio_pre} extends \eqref{eqn:bad_branching_ratio_pre} rather than
competing with it.

\smallskip
The following further examples explore the range of behaviour that $a_i(\ell)$
and these two limits can exhibit in more general networks. These examples motivate the results and proofs that follow.

\begin{example}\em
\label{ex:cycle}
Let $\GG$ be a simple $n$-cycle, as in Fig. \ref{F:n-cycle}.
This network, like the previous one, is strongly connected.

\begin{figure}[h!]
\centerline{%
\includegraphics[width=0.2\textwidth]{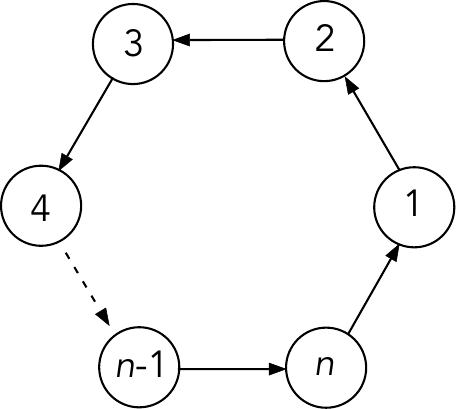}
}
\caption{A simple $n$-cycle.}
\label{F:n-cycle}
\end{figure}
Its adjacency matrix is
\[
A = \Matrix{0 & 1 & 0 & \cdots & 0 & 0 \\
0 & 0 & 1 & \cdots & 0 & 0 \\
\vdots & \vdots & \vdots & \ddots & \vdots & \vdots \\
0 & 0 & 0 & \cdots & 0 & 1 \\
1 & 0 & 0 & \cdots & 0 & 0 
}
\]
The eigenvalues are $\lambda_k = \zeta^k$ where $\zeta = \ee^{2\pi\ii/n}$
is a primitive $n$-th root of unity. All eigenvalues are simple, and all have absolute value $1$.
The unique real eigenvalue with maximal absolute value (the Perron eigenvalue) is $\xi=1$.
Clearly, $a_i(\ell) = 1$ for all $i, \ell$, so both limits \eqref{eqn:bad_branching_ratio_pre} and \eqref{eqn:good_branching_ratio_pre} exist and equal $1$.

This is an example showing that even when $\GG$ is strongly connected,
eigenvalues with maximal absolute value need not be unique.
\end{example}

\begin{example}\em
\label{no_ratio}

A first example where \eqref{eqn:bad_branching_ratio_pre} fails to exist is  node 1 in the strongly connected network depicted in
Fig.\ref{fig:6node} (left); the input
tree of node 1 is shown in Fig.\ref{fig:6node} (right): it represents all walks that terminate at node 1. The input tree is a rooted tree with root node 1, and the number of nodes at level $\ell$ of the tree is $a_1(\ell)$.

\begin{figure}[h!]
\centerline{%
\includegraphics[width=0.2\textwidth]{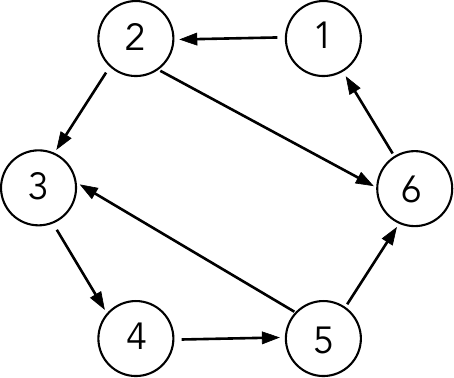}
\qquad \qquad 
\includegraphics[width=0.12\textwidth]{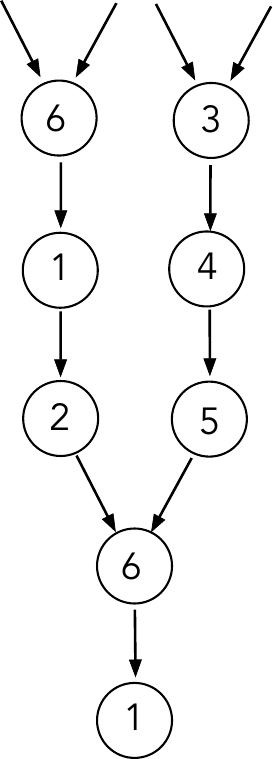}}
\caption{{\em Left}: A directed graph
with six nodes. {\em Right}: Input tree of node 1.
}
\label{fig:6node}
\end{figure}

The sequence $a_1(\ell)$ for $\ell = 0,1,2,\ldots$ can be easily seen from the input tree to be:
\[
1,1,2,2,2,4,4,4,8,8,8,\ldots
\]
The ratios $a_1(\ell+1)/a_1(\ell)$ are:
\[
1,2,1,1,2,1,1,2,1,1,2,\ldots
\]
which fails to converge. So $\rho(1)$, as defined by~\eqref{eqn:bad_branching_ratio_pre}, does not exist.

The adjacency matrix of the graph is:
\[
A = \Matrix{0 & 1 & 0 & 0 & 0 & 0 \\ 0 & 0 & 1 & 0 & 0 & 1 \\ 0 & 0 & 0 & 1 & 0 & 0 \\ 0 & 0 & 0 &
0 & 1 & 0 \\ 0 & 0 & 1 & 0 & 0 & 1 \\ 1 & 0 & 0 & 0 & 0 & 0} 
\]
Its eigenvalues are $\xi, \xi\zeta, \xi\zeta^2$ where $\xi=\sqrt[3]{2}$
and $\zeta = \ee^{2\pi\ii/3}$ is a primitive cube root of unity,
together with an eigenvalue 0 of multiplicity 3.
The first three eigenvalues all have the maximal absolute value $\xi$, and the first of
them is the Perron eigenvalue
$\xi=\sqrt[3]{2}$.

We can find the branching ratio $\delta(i)$ explicitly for all nodes $i$. 
The case $i = 1$ is typical: the other nodes give similar results.
The value $\xi=\sqrt[3]{2}$ makes intuitive sense 
since $a_1(\ell)$ doubles every three steps, so
$a_1(\ell)$ grows approximately as $2^{\ell/3}$. However, it is not
asymptotic to $2^{\ell/3}=\xi^\ell$, because the exact value of $a_1(\ell)$ is:
\begin{equation}
\label{E:exact}
a_1(\ell) = 2^{\lfloor (\ell+1)/3\rfloor}.
\end{equation}
By \eqref{E:exact}, 
\[
\frac{a_1(\ell)}{\xi^\ell} = \left\{ \begin{array}{lll} 
1 \ \ \, \quad {\rm if}\quad  \ell \equiv 0 \pmod{3}\\
\xi^{-1} \quad {\rm if}\quad  \ell \equiv 1 \pmod{3}\\
\xi  \qquad {\rm if}\quad  \ell \equiv 2 \pmod{3}
\end{array}\right. 
\]
Passing to an asymptotic expression (for later use) we have
\[
a_1(\ell) \sim \left\{ \begin{array}{ll} 
{\xi^\ell} &{\rm if}\quad  \ell \equiv 0 \pmod{3}\\
\xi^{-1}.{\xi^\ell} &{\rm if}\quad  \ell \equiv 1 \pmod{3}\\
\xi.\xi^{\ell}  &{\rm if}\quad  \ell \equiv 2 \pmod{3}
\end{array}\right. 
\]
Therefore, the limit in \eqref{eqn:good_branching_ratio_pre} exists and equals $\xi$:
\[
\delta(1)=\lim_{\ell \to \infty} (a_1(\ell))^{1/\ell} = 
\xi 
\]
in all three cases.

\end{example}

This example sheds some light on what happens when the maximal absolute value
eigenvalues are multiple (that is, $A$ has positive period $h > 1$, see Section \ref{S:PFT}).
It also shows that even when \eqref{eqn:bad_branching_ratio_pre} converges,
it can do so slowly, and the same goes for \eqref{eqn:good_branching_ratio_pre}.
%For computations to more than a couple of decimal places,
%it is better to compute the Perron eigenvalue using a computer
%algebra package.

It also shows that even when the network is strongly connected,
the asymptotic formula $a_i(\ell) \sim C\xi^\ell$ need not hold.
Here, the constant $C$ depends on $\ell \pmod{3}$.

The constants also depend on the node $i$. For example,
similar calculations show that $a_2(\ell) = 2^{\lfloor \frac{\ell}{3}\rfloor}$, so
\[
a_2(\ell) \sim \left\{ \begin{array}{ll} 
{\xi^\ell} &{\rm if}\quad  \ell \equiv 0 \pmod{3}\\
\xi^{-1}.{\xi^\ell} &{\rm if}\quad  \ell \equiv 1 \pmod{3}\\
\xi^{-2}.\xi^{\ell}  &{\rm if}\quad  \ell \equiv 2 \pmod{3}
\end{array}\right. 
\]

In Section \ref{S:asymp_h} we prove that this example
is typical when the period $h > 1$.

\begin{example}\em
\label{abff}
Next, we consider a network that is \emph{not} strongly connected, Fig.\ref{F:abff}.
 Here $\alpha,\beta$ are positive integers representing edge multiplicities.

\begin{figure}[h!]
\centerline{%
\includegraphics[width=0.3\textwidth]{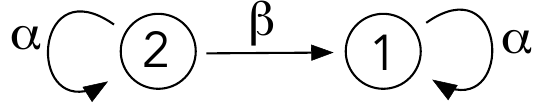}}
\caption{Network whose adjacency matrix is reducible.}
\label{F:abff}
\end{figure}

This example has adjacency matrix
\[
A = \Matrix{\alpha  & 0 \\ \beta & \alpha}
\]
which is obviously reducible. Indeed, the network has two strongly connected components.

There is still a unique maximal eigenvalue $\alpha$, but now it has
multiplicity $2$. The left generalised eigenvectors
can be taken to be $\vc e_1 = [1,0]$ and $\vc e_2 = [0,1]$. We have
\[
\vc e_1 A = [\alpha,0] = \alpha \cdot \vc e_1 \qquad \vc e_2A = [\beta,\alpha] = \beta \cdot \vc e_1+\alpha \cdot \vc e_2.
\]
The powers of $A$ can easily be shown to be:
\[
A^\ell = \Matrix{\alpha^\ell  & 0 \\ \ell \alpha^{\ell-1}\beta & \alpha^\ell}.
\]
Hence, the numbers of walks of length $\ell$ ending at nodes 1, 2 respectively are:
\[
 a_1(\ell) = \alpha^\ell  + \ell \alpha^{\ell-1}\beta  \qquad a_2(\ell) = \alpha^\ell .
\]
Clearly the branching ratio for node 2 is $\delta(2)=\alpha$. We claim that the
same holds for node 1. This follows since
\[
 a_1(\ell) = \alpha^\ell  + \ell \alpha^{\ell-1}\beta =\alpha^\ell\left(1+ \frac{\beta}{\alpha}\ell\right)
\]
so
\[
\frac{\log  a_1(\ell)}{\ell} = \frac{\ell \log \alpha}{\ell} + \frac{\log(1+ \frac{\beta}{\alpha}\ell)}{\ell}
\]
which tends to $\log \alpha$ as $\ell \to \infty$. Therefore $a_1(\ell)^{1/\ell}$
tends to $\alpha$, that is, $\delta(1)=\alpha$.
Note that the non-diagonal Jordan block ($\beta$) does not change the branching ratio.

We also have
\beqn
\rho(1)=\frac{a_1(\ell + 1)}{a_1(\ell)} &=& \frac{\alpha^{\ell+1}+(\ell+1)\alpha^\ell \beta}{\alpha^\ell+\ell \alpha^{\ell-1}\beta}
	\to \alpha \\
\rho(2)=\frac{a_2(\ell + 1)}{a_2(\ell)} &=& \frac{\alpha^{\ell+1}}{\alpha^\ell} = \alpha
\eeqn
as $\ell \to \infty$. So here the sequence of successive ratios converges,
and it converges to the same value for both nodes, even though $A$ is reducible.
As anticipated, $\rho(i)=\delta(i)$ for both nodes $i=1,2$; this is a general fact (Remark \ref{R:ratio_vs_root}): when $\rho(i)$ exists, it equals $\delta(i)$.
\end{example}

\begin{example}\em
\label{ex:8nodeFF}
The next example illustrates how the asymptotics of $a_i(\ell)$
can behave for reducible $A$ of period $h > 1$. Fig. \ref{F:8nodeFF}
shows two 4-cycles linked by one edge. The graph has period $h=4$.
With a suitable ordering of the nodes (precisely, 56781234), the adjacency matrix is block upper triangular, 
\[
A =\left[ \begin{array}{cccc|ccccc}
0 & 0& 0& 1& 0& 0& 0&  1\\
1& 0& 0& 0& 0& 0& 0& 0 \\
0& 1& 0& 0& 0& 0& 0& 0 \\
0& 0& 1& 0& 0& 0& 0& 0 \\
\hline
0& 0& 0& 0& 0& 0& 0& 1 \\
0& 0& 0& 0& 1& 0& 0& 0 \\
0& 0& 0& 0& 0& 1& 0& 0 \\
0& 0& 0& 0& 0& 0& 1& 0 
\end{array}\right]
\]
Its eigenvalues are
\[
1, \ii, -1, -\ii
\]
each of multiplicity 2. Absolute values are all $1$, and the Perron eigenvalue is $1$.

\begin{figure}[b]
\centerline{%
\includegraphics[width=0.35\textwidth]{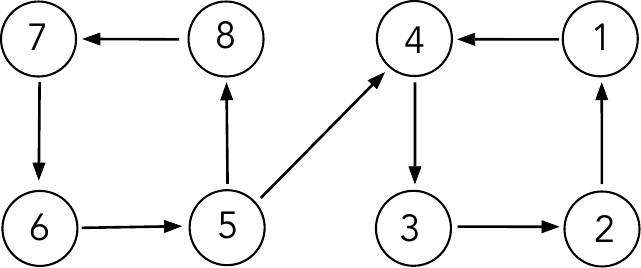}}
\caption{An 8-node network with reducible adjacency matrix and $h=4$. }
\label{F:8nodeFF}
\end{figure}

It can be shown inductively that (for example)
\[
a_4(\ell) = \left\lceil \frac{\ell}{4} \right\rceil + 1
\]
%where $\lceil \cdot \rceil$ is the ceiling function.
or, more precisely,
\begin{equation}
\label{E:8nodeFF}
a_4(\ell) = \left\{
\begin{array}{rcl}
\frac{\ell}{4} + 1 & {\rm if} & \ell \equiv 0 \pmod{4} \\
\frac{\ell}{4} + \frac{7}{4} & {\rm if} & \ell \equiv 1 \pmod{4} \\
\frac{\ell}{4} + \frac{3}{2}  & {\rm if} & \ell \equiv 2 \pmod{4} \\
\frac{\ell}{4} + \frac{5}{4}  & {\rm if} & \ell \equiv 3 \pmod{4}.
\end{array}
\right.
\end{equation}
Therefore the asymptotic behaviour of $a_4(\ell)$ is of the form
$P_k(\ell) \xi^\ell$ where the polynomials $P_k$
depend on $k=\ell \pmod{4}$. No single polynomial factor works.
\end{example}

\begin{example}\em
\label{E:polycycle}
Let us now consider a class of strongly connected networks that generalises Example~\ref{ex:cycle}.
Let $\GG$ have $n$ nodes, forming a unidirectional cycle in
which the edge from node $i+1$ to $i$ has multiplicity $m_i \geq 1$.
(Identify node 1 with node $n+1$ here.) We call such a network
an $(m_1,\ldots,m_n)$-{\em polycycle} (Fig. \ref{F:polycycle}).
$\GG$ is strongly connected, and $A$ is irreducible of period $n$.

\begin{figure}[h!]
\centerline{%
\includegraphics[width=0.2\textwidth]{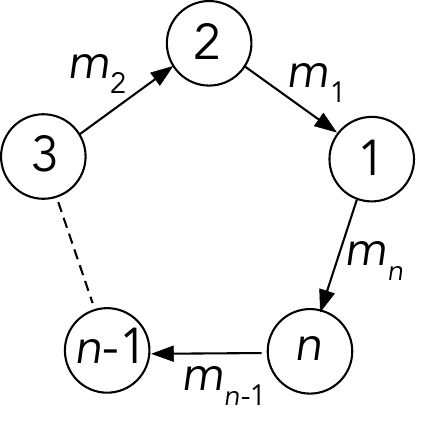}}
\caption{An $n$-node polycycle.}
\label{F:polycycle}
\end{figure}

Let $M = m_1 m_2 \cdots m_n$. It is easy to see that
if $p \geq 0, 0\leq q\leq n-1$, then
\[
a_1(pn+q) = M^p m_1\cdots m_q
\]
The Perron eigenvalue is $\xi = M^{1/n}$. Therefore
\[
a_1(pn+q) = \xi^{n p} m_1\cdots m_q = C_q  \xi^{pn+q}
\]
where
\[
C_q = \frac{m_1\cdots m_q}{\xi^q}.
\]
Hence,
\[
a_1(\ell) = C_{\ell \bmod n} \xi^\ell.
\]
Moreover,
\[
C_{q}/C_{q-1} = m_q/\xi \quad {\rm when}\ 1 \leq q \leq n-1
\]
So the ratios of successive $C_q$ can differ, in a way that 
depends on the sequence $(m_i)$. In particular, unless all the $m_i$ are equal,
\[
    \rho(i)=\lim_{\ell \to \infty} a_i(\ell+1)/a_i(\ell)
\]
does not exist.
Nonetheless, $\delta(i)$ does exist: for instance,
\[
   \delta(1)=\lim_{\ell\to\infty} \left(a_1(\ell)\right)^{1/\ell}=\xi \lim_{\ell \to\infty} C_{\ell \bmod n}^{1/\ell} =\xi
\]
\end{example}

\begin{example}\em
\label{ex:3nodeFF}

Finally we consider an example where $A$ is reducible
and branching ratios are different for each node, shown in Fig.\ref{F:3nodeFF}.

\begin{figure}[h!]
\centerline{%
\includegraphics[width=0.3\textwidth]{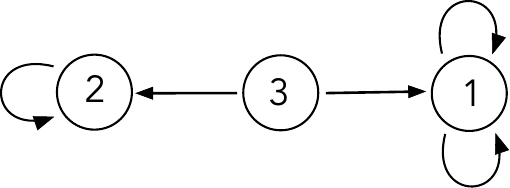}}
\caption{3-node network with distinct branching ratios.}
\label{F:3nodeFF}
\end{figure}

The adjacency matrix is
\[
\Matrix{2&0&0\\0&1&0\\1&1&0}
\]
with eigenvalues $2,1,0$. The matrix is reducible. Indeed, $\GG$ has three
strongly connected components, one for each node.
It is easy to see that for all $\ell \geq 1$,
\beqn
a_1(\ell) &=& 3\cdot 2^{\ell-1} \\
a_2(\ell) &=& 2 \\
a_3(\ell) &=& 0
\eeqn
Node 1 has branching ratio 2, node 2 has branching ratio 1, and node 3 
has branching ratio 0.
\end{example}

\subsection{Biological Background} 

Makse and coauthors have applied branching ratios to biological
networks, in particular to gene regulatory networks \cite{LMRASM20,MLM20} and connectomes
\cite{MM19}. The theory and applications of these ideas are treated 
comprehensively in \cite{MBSS25}. The input tree quantifies the passage
of information through the network to the given node, and the branching ratio is
an important isomorphism invariant of the input tree. The main issue is the
existence (or not) of {\em cluster synchrony} in the network, which occurs
when certain sets of nodes have identical dynamics in any system
of ordinary differential equations (ODEs) that respects the network topology.

Cluster synchrony is governed by the
essentially equivalent notions of a {\em graph fibration} \cite{BV02,BV02a}, 
a {\em balanced colouring} \cite{GS23}, and an {\em equitable partition} \cite{BV02,S74}.
Any of these essentially equivalent concepts determines the possible robust
synchrony patterns in networks of coupled ODEs.
A {\em colouring} of the nodes assigns to each node $i$ an element $\kappa(i)$
of a set of {\em colours} $\KK$. The colouring
is {\em balanced} if nodes with the same colour
have colour-isomorphic input sets. A graph fibration
is a graph homomorphism $\phi:\GG \to \HH$ that preserves input sets, and its fibres
(the sets $\phi^{-1}(i)$ for $i \in \HH$) determine a colouring. A
{\em partition} divides the nodes into disjoint subsets
according to their colours, and the partition is {\em equitable} if the
corresponding colouring is balanced. The subsets of
nodes with a given colour are called {\em clusters}; a cluster is the same
as a fibre of a fibration or a part of an equitable partition.

Among all balanced colourings there is a unique coarsest one \cite{S07b}, also called
the {\em minimal equitable partition}. It can be characterised by the topological
type of the infinite input tree:
the coarsest balanced colouring is the one for which nodes have the same colour
if and only if they have isomorphic input trees \cite{BV02,S07b}. The minimal equitable partition
determines the most highly synchronised set of clusters; that is, the one
with fewest colours.

\subsection{Distributed Systems}

The role of graph fibrations in the computational power of distributed systems was highlighted in the 1990s~\cite{BV96, BV97}: the idea is that if all processors in a synchronous distributed system start from the same state and have the same input tree, they will always be in the same state in the future, whichever protocol they are executing. Hence, the shape of the input tree essentially constrains and determines the behaviour of processors in a computational network. Of course, if we relax some assumption (for instance, if we allow messages to take some unknown time to travel along the edges, or if we assume that the computation happens asynchronously, or if we allow nodes to start from different states) the above condition becomes true only \emph{in the worst case}---we can never assume for certain that nodes with the same input tree will behave in a different way.

Fibrations (the morphism-theoretic view of equitable partitions, essentially
equivalent to balanced colourings) become the natural way to describe how the behaviour of a network $\GG$ can be simulated on a (typically: smaller) network $\HH$, having a different topology but the same input trees. When this simulation game is brought to its limit, we are considering the coarsest equitable partition, leading to a minimal fibration, and $\HH$ in this case is the minimum fibration base of $\GG$.
The minimum fibration base is the smallest graph having the same set of input trees as $\GG$.
It is equivalent to the `quotient network' of \cite{GS23}.

In a distributed system, the input tree of node $i$ represents how messages coming from  nodes of the network can reach $i$; more precisely, every walk of length $\ell$ starting at $j$ and ending at $i$ can be seen as a message that is sent by $j$ and is delivered to $i$ after $\ell$ passages. In this sense, the number of walks of length $\ell$ ending at $i$ is the (maximum) number of messages that $i$ can receive at the $\ell$-th round of a synchronous computation. The growth rate corresponds intuitively to how this quantity grows as $\ell$ goes to infinity. 

It may seem that this asymptotic behaviour is not really relevant: after all, Dana Angluin~\cite{A80} showed that in a network of $n$ processors the first $n-1$ levels of the input tree determine the rest. But there are situations in which $n$ is unknown to the processors, or in which the initial state is arbitrarily corrupted: having explicit results on the asymptotic behaviour becomes crucial in such settings.

\subsection{Summary of Paper and Main Results}

Section \ref{S:GTP} sets up the main graph-theoretic notions required
and the properties that we use:
adjacency matrices, walks, input trees, strongly connected components,
the upstream subnetwork of a node, a formula for the number $a_i(\ell)$
of walks of length $\ell$ terminating at node $i$, reducible and irreducible matrices.

The most obvious definition of the branching ratio is the limit of the ratios $a_i(\ell+1)/a_i(\ell)$ as
$\ell$ tends to infinity. As observed in Section \ref{S:ME},
this definition fails for some multigraphs. We therefore
propose a different definition of
the branching ratio $\delta(i)$, which exists for all finite multigraphs, and establish some basic properties.
For convenience, and to clarify this summary, we re-state the definition here: 
\begin{equation}
\label{E:BR_def}
 \delta(i) =
  \lim_{\ell \to \infty} (a_i(\ell))^{1/\ell}
  \end{equation}
At this stage we do not know that this sequence converges, but
Lemma \ref{L:bounds} proves that the limit 
exists provided certain upper and lower bounds hold, and expresses
that limit in terms of those bounds.
Lemma \ref{L:upper_bound} gives an upper bound
$a_i(\ell) \leq P(\ell)\xi(i)^\ell$,
where $P(\ell)$ is an eventually positive polynomial in $\ell$.
Thus it remains only to prove a suitable lower bound on $a_i(\ell)$.
Proposition \ref{P:upstream_principle} proves
the `upstream principle', which is used later to obtain the required 
lower bound on $a_i(\ell)$.
The motivating examples of Section \ref{S:ME} are small
networks for which $a_i(\ell)$ is computed explicitly. These examples
motivate the main theorems and the asymptotic expressions
for $a_i(\ell)$ derived in Section \ref{S:A}.

Section \ref{S:PFT} recalls relevant parts of the Perron--Frobenius Theorem,
which is the key to all subsequent results. We define $\xi(i)$ to be the
Perron eigenvalue (spectral radius) of the adjacency matrix of the
upstream subnetwork $\UU(i)$.

The main theorems about the branching ratio
are then derived in two stages.

In Section \ref{S:BRSCNP} we prove
that for any strongly connected network, of any period, 
the limit in \eqref{E:BR_def} converges
and is equal to the Perron eigenvalue $\xi(i)$ (Theorem \ref{T:BRexists_any_h}). The proof is
based on a formula for a Ces\`aro average, which is part of the 
Perron--Frobenius Theorem.

Section \ref{S:BRGC1} uses Theorem \ref{T:BRexists_any_h} 
and the Upstream Principle to remove the assumption that the network
is strongly connected. Specifically, Theorem \ref{T:gencase} proves 
that if $\GG$ is any network then the limit in \eqref{E:BR_def} converges
to $\xi(i)$.
This completes the analysis of the branching ratio.

Section \ref{S:A} extends the results in a different direction. We refine the
results about the branching ratio to give more precise asymptotics for
$a_i(\ell)$ as $\ell \to \infty$. We develop the results in two steps:  
first for any strongly connected network, then for
any network. These asymptotic expressions respectively involve a single constant,
a finite set of constants depending on $\ell \pmod{h}$, and a
finite set of polynomials depending on $\ell \pmod{g}$, where $g$ is the
lcm of the periods of strongly connected components
 with Perron eigenvalue equal to $\xi(i)$.

\section{Graph-Theoretic Preliminaries}
\label{S:GTP}

Different areas of application use different terminology and notation for the
same graph-theoretic concepts. We therefore clarify our usage here.

\subsection{Graph Terminology and Notation}

Let $\GG =(\VV,\EE,s,t)$ be a finite directed multigraph, that is, a directed graph in which multiple edges and self-loops are permitted. For simplicity
we call such a graph a {\em network}. Here, $\VV$ is the set
of {\em vertices} (nodes) and $\EE$ is the set of (directed) {\em edges} (arrows, arcs).
If $e \in \EE$ the map $s:\EE\to\VV$ determines the source (tail) $s(e)$ of $e$, and the map $t:\EE\to\VV$ specifies the target (head) $t(e)$ of $e$.
 
Given a set of vertices $X\subseteq\VV$, we write $\GG[X]$ for the subgraph {\em induced} by $X$, that is, the graph having the elements of $X$ as vertices, and all arcs with both source and target in $X$.

The {\em adjacency matrix} of a network $\GG$ with $n$ nodes
 is the $n \times n$ matrix $A = (a_{ij})$ such that $a_{ij}$
 is the number of edges from node $i$ to node $j$.

Obviously, the adjacency matrix $A$ depends on the way in which the nodes of $\GG$ are ordered: different orderings of the nodes will induce different permutations of (the rows and columns of) $A$.
This fact will have relevance in some of the statements.
This definition is convenient for working with row vectors, which we do here. Some sources use the transpose of this matrix, which is more convenient for network dynamics since column vectors are commonly used there. 

\begin{definition}\em
\label{D:spectrum}
The {\em spectrum} of $A$ (and of $\GG$) is the set of eigenvalues
of $A$. The {\em spectral radius} of $A$ (and of $\GG$) is the maximum 
absolute value of the eigenvalues of $A$. Since $A$ is a non-negative matrix, the spectral radius itself is also an eigenvalue (although not always the only eigenvalue of maximum modulus), usually called the {\em Perron eigenvalue}.
\end{definition}

A {\em walk} in $\GG$ {\em from} node $j$ {\em to} node $i$
is a sequence of nodes $x_t$ and edges $e_t$ of the form
\[
(x_0, e_1, x_1, e_2,x_2, \ldots, e_\ell, x_\ell)
\]
where $x_0=i$ and $x_\ell = j$,
 such that
$x_k = t(e_{k+1})$ for all $0 \leq k \leq \ell-1$ and
$x_k = s(e_k)$ for all $ 1 \leq k \leq \ell$.
The {\em length} of the walk is $\ell$.

Schematically, a walk from $j$ to $i$ looks like this:
\[
i=x_0\stackrel{e_1}{\leftarrow} x_1\stackrel{e_2}{\leftarrow}x_2 \leftarrow\cdots\leftarrow x_{\ell-1} \stackrel{e_\ell}{\leftarrow} x_\ell=j.
\]
We number the nodes and edges in this order for convenience,
since we think of constructing such walks backwards, starting from the final target 
node and iterating input edges.

The sequence of edges $(e_j)$ determines the walk uniquely,
except in the trivial case where $\ell = 0$, when the walk consists of
node $x_0=i=j$ alone. However, it is more clear and slightly more convenient to include the nodes in the definition. Observe, incidentally, that in a multigraph, the sequence of nodes does \emph{not} determine the walk uniquely, since there may be multiple edges between two nodes.

%In the network dynamics literature a walk is often termed a `path'. In graph theory a `path'
%is usually defined to be a walk without repeated nodes or edges. A `trail'
%has no repeated edges. Counting paths or trails is more difficult than counting walks,
%and here we focus solely on walks.

\begin{definition}\em
\label{D:number_of_walks}
Let $i$ be a node in a network $\GG$. Then we write
\[
a_i(\ell) = \mbox{number of walks of length $\ell$ terminating at $i$}.
\]
\end{definition}
This is the main quantity that we study in this paper, focusing
on how $a_i(\ell)$ grows as $\ell$ increases indefinitely.

\subsection{Input Trees}

The walks of given length that terminate at node $i$ can be organised into
the {\em input tree} of that node~\cite{BV02,GS23,S07b}. Each node in
the input tree corresponds bijectively to a walk terminating at $i$,
which is the unique walk from that node of the input tree to the tree root, which corresponds to node $i$.

For instance, the input
trees for Example \ref{ex:fibo} are shown in Fig.\ref{F:fibo_input_trees}.
The walk corresponding to node 1 at the top of the left-hand tree
is $1 \leftarrow 2 \leftarrow 1 \leftarrow 2 \leftarrow 1$, the leftmost walk down the tree, and so on
for all other nodes of these input trees.
See \cite{S07b} for details.

In this case, \eqref{E:fibo_asymp} shows that the number $\phi$ can be interpreted as the asymptotic growth rate of the number of walks $a_i(\ell)$ as $\ell \to\infty$; that is, the `branching ratio' of the input tree from one level to the next.

\begin{figure}[h!]
\centerline{%
\includegraphics[width=0.7\textwidth]{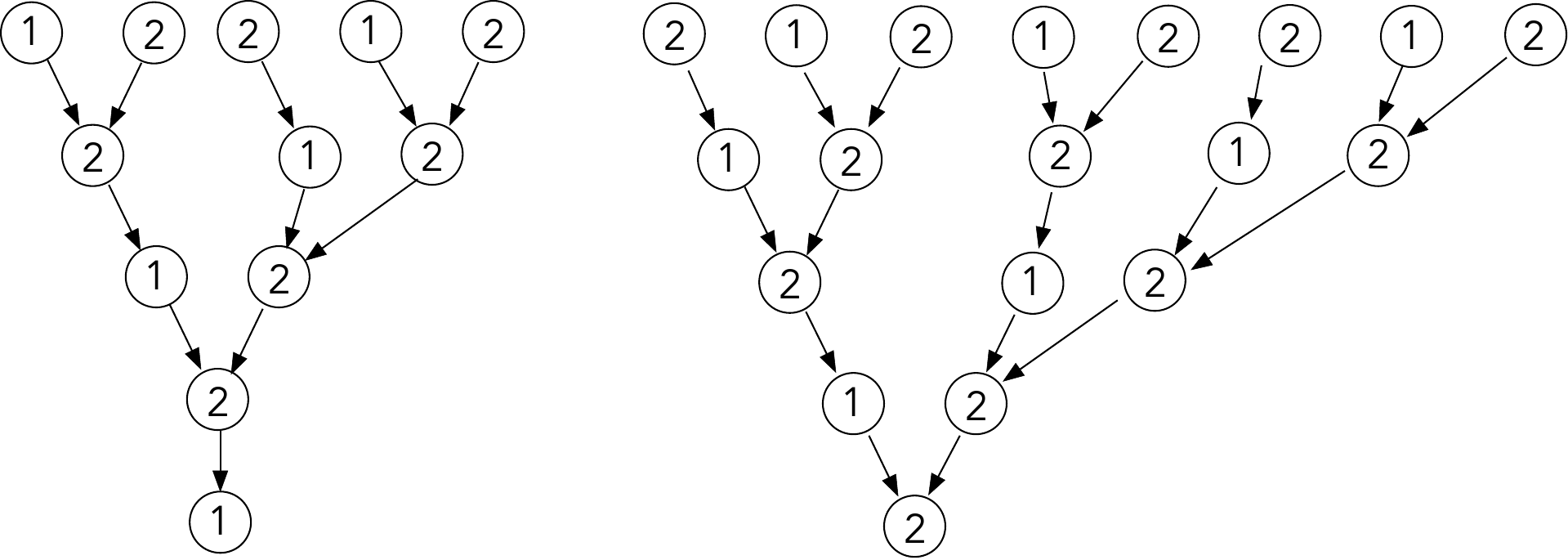}}
\caption{Input trees for the Fibonacci circuit, up to level $\ell= 4$.}
\label{F:fibo_input_trees}
\end{figure}

\subsection{Upstream Subnetworks and Strongly Connected Components}

We say that a network $\GG$ is {\em strongly connected} if, for any two
distinct nodes $i,j$, there is a walk from $i$ to $j$.
A {\em strongly connected component} ({\em SCC}) of $\GG$
is a maximal set $X$ such that $\GG[X]$ is strongly connected.

Every network $\GG$ decomposes uniquely into 
strongly connected components. These components correspond to the nodes of the {\em condensation} of $\GG$,  which has a (unique) edge between components if and only if there is an edge between  some members of those components. The condensation is acyclic.

\smallskip
The branching ratio of a given node can, of course, depend on the node.
The crucial feature that determines this dependence is:

\begin{definition}\em
\label{D:upstream_sub}
The {\em upstream subnetwork} $\UU(i)$ of node $i \in \VV$ is the
induced subnetwork $\GG[\XX_i]$, where $\XX_i$ is
the set of all vertices $j$ such that there is a walk from $j$ to $i$.
\end{definition}

As we shall see below (Proposition~\ref{P:upstream_properties}), it is always possible to (re)order the nodes of a multigraph $\GG$ in such a way that if $i$ and $j$ belong to two different SCCs, but there is a walk from $i$ to $j$, then $i<j$. (This is not possible within a single SCC unless it contains only one node.) Any such ordering is said to be \emph{compatible with the feedforward structure} of $\GG$. The adjacency matrix of $\GG$ with this order has a block upper triangular structure.

\begin{example}\em
\label{ex:upstream}

Fig. \ref{F:upstream} shows two networks, each composed of
a Fibonacci circuit and a single node with two self-loops.
These subnetworks are connected in a feedforward manner, 
in two different orders.

\begin{figure}[h!]
\centerline{%
\includegraphics[width=0.5\textwidth]{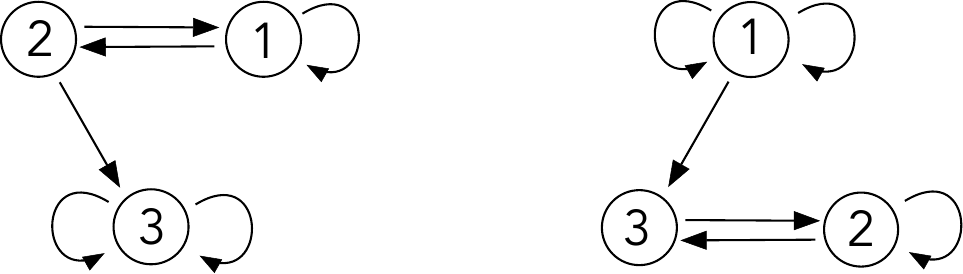}}
\caption{{\em Left}: Fibonacci circuit feeding forward to a single node with two
self-loops. {\em Right}: Single node with two
self-loops feeding forward to a Fibonacci circuit.}
\label{F:upstream}
\end{figure}

{\em Left Network}: The SCCs are $\{1,2\}$ and $\{3\}$. 
The natural ordering on $\{1,2,3\}$ is compatible with the feedforward structure.
As a consequence, its adjacency matrix is
\[
A = \left[
\begin{array}{cc|c}
1 & 1 & 0 \\
1 & 0 & 1 \\
\hline
0 & 0 & 2
\end{array}
\right]
\]
and it has a block upper triangular structure. 
The eigenvalues of $A$ are $\frac{1\pm\sqrt{5}}{2}$ and $2$.
The upstream subnetworks are
\[
\UU(1) = \UU(2) = \GG[\{1,2\}] \qquad  \UU(3) = \GG.
\]
The branching ratios are
\[
\delta(1)=\delta(2)= \frac{1+\sqrt{5}}{2} \qquad \delta(3) = 2.
\]

{\em Right Network}: The SCCs are $\{1\}$ and $\{2,3\}$. 
Again, the natural ordering on $\{1,2,3\}$ is compatible with the feedforward structure.
The adjacency matrix is
\[
A = \left[
\begin{array}{c|cc}
2 & 0 & 1 \\
\hline
0 & 1 & 1 \\
0 & 1 & 0
\end{array}
\right]
\]
once more exhibiting the stated block upper triangular structure. 
The eigenvalues of $A$ are again $\frac{1\pm\sqrt{5}}{2}$ and $2$.
The upstream subnetworks are
\[
\UU(1) = \GG[\{1\}] \qquad \UU(2) = \UU(3) = \GG.
\]
The branching ratios are easily seen to be
\[
\delta(1) = \delta(2)= \delta(3)= 2
\]
\end{example}

\noindent
In these examples:
\begin{itemize}
\item[\rm (a)] The branching ratio $\delta(i)$ can depend on $i$.
\item[\rm (b)] The branching ratio $\delta(i)$ need not equal the spectral radius of $\GG$.
\item[\rm (c)] The branching ratio $\delta(i)$ is the spectral radius of the upstream
subnetwork $\UU(i)$.
\end{itemize}
We will see that these features are typical of general networks.

We now state a lengthy list of properties of upstream subnetworks.
The proofs are straightforward and are omitted, but the results are
needed repeatedly in the proofs of the main theorems.

\begin{proposition}
\label{P:upstream_properties}
$ $

{\rm(1)} If $(x_0,e_1, x_1, e_2,x_2, \ldots, e_\ell, x_\ell)$ is a walk from $j=x_\ell$ to $i=x_0$
then all $x_t$ and $e_t$ are included in $\UU(i)$.

{\rm(2)} $\UU(i)$ consists precisely of the vertices and edges of all walks
that terminate at vertex $i$.

{\rm(3)} If $b_i(\ell)$ is the number of walks in $\UU(i)$ of length $\ell$ that terminate
at vertex $i$, then
\[
b_i(\ell) = a_i(\ell) \quad \mbox{\rm for all}\ i, \ell.
\]

{\rm(4)} $\UU(i)$ is the subnetwork induced by the set of all the vertices of all SCCs of $\GG$ that are contained in $\UU(i)$.

{\rm(5)} Every SCC of $\UU(i)$ is an SCC of $\GG$.

{\rm(6)} If we arrange the SCCs of $\GG$ in a total order that is compatible
with the feedforward structure of the set of SCCs, and order nodes by blocks,
then $A$ is block upper triangular.

{\rm(7)} Consider an ordering of the nodes of $\GG$ respecting the previous statement, and let $B(i)$ be the adjacency matrix of $\UU(i)$. Then $B(i)$ consists of
those blocks in $A$ that correspond to rows and columns indexed by
the SCCs of $\UU(i)$. Therefore $B(i)$ inherits the block upper triangular
structure:
\begin{equation}
\label{E:block_B}
B(i) =  \Matrix{B_1 & \ast & \ast  & \cdots & \ast &\ast \\
	0 & B_2 & \ast  &  \cdots & \ast& \ast \\
	\vdots & \vdots  & \vdots &\ddots & \vdots& \vdots \\
	0 & 0 & 0  & \cdots & B_{p-1}& \ast\\
	 0 & 0 & 0 & \cdots & 0 & B_p
}
\end{equation}
for suitable $p$. The blocks $B_j$ depend on the choice of $i$.

{\rm(8)} The spectrum of $B(i)$ is the union of the spectra of the $B_j$ in \eqref{E:block_B}.

{\rm(9)} The spectrum of $B(i)$ is a subset of the spectrum of $A$.
\qed\end{proposition}

Recall Definition \ref{D:spectrum} for the  Perron eigenvalue (or spectral radius):

\begin{proposition}
\label{P:same_P_ev}
If $i$ and $j$ lie in the same SCC of $\GG$ then the Perron
eigenvalues $\xi(i)$ of $\UU(i)$ and $\xi(j)$ of $\UU(j)$ are equal.
\end{proposition}
\begin{proof}
The upstream subnetworks $\UU(i)$ and $\UU(j)$ are equal, and
the Perron
eigenvalues $\xi(i)$ and $\xi(j)$ are equal to the maximal eigenvalue of
this common upstream subnetwork.
\end{proof}

\subsection{Number of Walks of Given Length}

The results here are standard: see for example \cite{AO,B73, BH90}, but
they are also vital to the analysis of the branching ratio since they relate it
to a suitable adjacency matrix.

Let $\GG$ be a network with nodes $\{1,2,\ldots,n\}$. Let
$A=(a_{ij})$ be the adjacency matrix of $\GG$. Then column $j$ of $A$
lists the numbers of edges from nodes $1,2,\ldots, n$ to $j$ ({\em input edges}).
The total number of input edges to node $j$ is therefore
$a_{1j}+a_{2j}+ \cdots + a_{nj}$.
Let ${\bf u} = (1,1,\ldots,1)$, considered as a row vector, that is, the $1 \times n$ matrix
\[
[1\ 1\ 1\ \ldots 1]
\]
and recall Definition \ref{D:number_of_walks}:
the number of walks of length $\ell$ terminating at $i$ is denoted by $a_i(\ell)$.
Well known ideas lead to a formula for $a_i(\ell)$ that is fundamental in
what follows:

\begin{lemma}
\label{L:in_sum}
$ $

{\rm (a)}
The total number $a_i(1)$ of input edges (equivalently, walks of length $1$) to node $i$ is $({\bf u} A)_i$.

{\rm (b)}
The total number of input walks of length $\ell$ to node $i$ is $a_i(\ell)=({\bf u} A^\ell)_i$.
\end{lemma}
\begin{proof}
Part (a):
\[
({\bf u} A) = [a_{11}+\cdots +a_{n1}, a_{12}+\cdots +a_{n2}, \ldots, a_{1n}+\cdots +a_{nn}],
\]
and the $i$-th entry is precisely the total number of input edges to node $i$.

Part (b): 
This follows  from (a) by an easy  induction on $\ell$. See \cite[Section 4.5 Corollary 1]{B73} or \cite[Section 1.3]{AO}.
\end{proof}

\subsection{Reducible and Irreducible Matrices}

\begin{definition}\em
\label{D:irreducible}
A  square (real) matrix $M$
is {\em irreducible} if it is nonzero and either of the following equivalent properties
 holds:

(a) There does not exist a permutation matrix $P$ such that
\begin{equation}
\label{E:reducible}
PMP^{-1} = \Matrix{R & S \\ 0 & T}
\end{equation}

(b) The directed graph $\GG$ with adjacency matrix $M$ (that is, $\GG$ has an
edge of multiplicity $m_{ij}$ from node $i$ to node $j$) is strongly connected.
\end{definition}

A {\em reducible} matrix is one that is not irreducible. 
If \eqref{E:reducible} holds for some permutation matrix $P$, the matrix $M$ is reducible.

Property (b) says that irreducibility of a matrix and strong connectivity
of a graph are the same concept viewed in two different ways. We keep the
two terms separate: throughout the paper, `strongly connected' refers to
graphs and `irreducible' refers to matrices.

\section{Definition of Branching Ratio}
\label{S:BR}

In this section we motivate a definition of the branching ratio that
applies to any network.

In simple examples, such as Example \ref{ex:fibo}, the sequence $a_i(\ell)$ 
approximates a geometric progression,
and the branching ratio can be defined as the common ratio of this progression,
which we denote by  $\rho(i)$. Specifically,
\begin{equation}
\label{E:ratio}
\rho(i) = \lim_{\ell \to \infty} \frac{a_i(\ell+1)}{a_i(\ell)}.
\end{equation}
This definition suffices if the limit exists, but
there are simple examples in which it does not, such as
Example \ref{no_ratio}. 
Nevertheless, these networks still have a well-defined branching ratio. 
We give a more general definition of the branching ratio
that applies to any network, prove that it always exists and
is equal to the maximal eigenvalue of the adjacency matrix
for the upstream subgraph of the node concerned.
The analysis relies on the 
Perron--Frobenius Theorem, which implies that such an eigenvalue exists
and is real. This characterisation reduces the computation of the
branching ratio to (numerical) linear algebra.

To dispose of (essentially trivial) exceptional cases, we need
the next proposition, which is well known and easy to prove:
\begin{proposition}
\label{P:acyclic} 
Let $\GG$ be a network with adjacency matrix $A$.
The following properties are equivalent:

{\rm (a)}: $\GG$ is acyclic (or feedforward).

{\rm (b)}: $A$ is nilpotent.

{\rm (c)}: All eigenvalues of $A$ are zero.

{\rm (d)}: Every input tree of $\GG$ is finite.
\qed\end{proposition}
Acyclic networks are therefore those networks for which
$a_i(\ell) = 0$ for sufficiently large $\ell$ and for all nodes $i$. In this case \eqref{E:ratio}
makes no sense, but by convention we define the
branching ratio to be zero. 

The general definition of the branching ratio proposed here,
which we distinguish from $\rho(i)$ by calling it $\delta(i)$, is:

\begin{definition}\em
\label{D:branching2} 
Let $\GG$ be a network and $i$ one of its nodes. The {\em branching ratio
 $\delta(i)$ at node $i$} is defined by:
\begin{equation}
\label{E:d}
\delta(i) =  \lim_{\ell \to \infty} (a_i(\ell))^{1/\ell}.
\end{equation}
\end{definition}

The first main result of this paper is a proof that the limit in \eqref{E:d}
always exists. If $\UU(i)$ is acyclic, the limit is $0$ since $a_i(\ell) = 0$
for sufficiently large $\ell$.

\begin{remark}\em
\label{R:ratio_vs_root}
Definition \ref{D:branching2} is a genuine extension of \eqref{E:ratio}, not a
competing alternative. If $\UU(i)$ is acyclic, the ratios in \eqref{E:ratio} are
eventually $0/0$, hence undefined. Otherwise $a_i(\ell) \geq 1$ for all $\ell$. 
%choose a cycle $\CC$ in $\UU(i)$ and a walk $W$ from a node of $\CC$ to $i$, of
%length $w$; a walk of length $\ell$ terminating at $i$ is given by the final
%$\ell$ edges of $W$ when $\ell \leq w$, and by a walk of length $\ell - w$ along
%$\CC$ followed by $W$ when $\ell > w$. 
Thus all the ratios in \eqref{E:ratio}
are defined and positive, and the standard inequalities
\[
\liminf_{\ell\to\infty} \frac{a_i(\ell+1)}{a_i(\ell)} \leq
\liminf_{\ell\to\infty} (a_i(\ell))^{1/\ell} \leq
\limsup_{\ell\to\infty} (a_i(\ell))^{1/\ell} \leq
\limsup_{\ell\to\infty} \frac{a_i(\ell+1)}{a_i(\ell)},
\]
valid for every positive sequence \cite[Theorem 3.37]{R76}, show that whenever
the limit $\rho(i)$ of \eqref{E:ratio} exists, the limit $\delta(i)$ of
\eqref{E:d} also exists and $\delta(i) = \rho(i)$. The converse fails:
in Example \ref{no_ratio} the limit $\delta(i)$ exists but $\rho(i)$ does not.
\end{remark}

Definition \ref{D:branching2} is motivated by an obvious upper bound on
$a_i(\ell)$, stated and proved in
Lemma \ref{L:upper_bound}, combined with the intuition that 
the key case for a suitable lower bound is that of a nontrivial
strongly connected network.
We may assume that the network is
not acyclic since \eqref{E:d} gives a branching ratio of $0$; this is
also the Perron eigenvalue by Proposition \ref{P:acyclic}.
In this case
there plausibly exist constants $C, \delta > 0$ such that
\begin{equation}
\label{E:branch_eq}
a_i(\ell) \sim C \delta^\ell \qquad \mbox{as}\ \ell \to \infty.
\end{equation}
This immediately implies \eqref{E:d}.
The same result holds if we replace $C$ by an eventually 
positive polynomial in $\ell$.

\begin{definition}\em
\label{D:perron_ev}
Let $\GG$ be a graph. We use $\xi(\GG)$ (or just $\xi$, when the context is clear) to denote the Perron eigenvalue of $\GG$. 
If $i$ is a node of $\GG$, we write $\xi(i)$ for the Perron eigenvalue of the upstream subgraph $\UU(i)$.
\end{definition}

Our second main result is that $\delta(i) = \xi(i)$ for any network. The next few sections set up necessary background and
give the proof.

\subsection{Bounds}

As remarked, there are three ways to think about branching ratios:
bounds, topology, and eigenvalues. We begin with observations
about suitable bounds on $a_i(\ell)$, which are motivated
by examples in Section \ref{S:ME} and later calculations.

\begin{definition}\em
\label{D:eventually}
A real polynomial $P(\ell)$ in a variable $\ell$ is {\em eventually positive}
if there exists $L$ such that $P(\ell)>0$ for all $\ell \geq L$.

Equivalently, $P(\ell)$ is eventually positive if the coefficient of the highest power of $\ell$ is positive.
\end{definition}

It is convenient to work with upper and lower bounds on $a_i(\ell)$
of the following form.

\begin{lemma}
\label{L:bounds}
Let $\GG$ be a graph and let $i$ be a node of $\GG$. Suppose that 
\begin{eqnarray}
\label{e:lower_bound} a_i(\ell) &\geq& Q(\ell)\delta^\ell\ \text{\rm for sufficiently large}\ \ell \\
\label{e:upper_bound} a_i(\ell) &\leq& P(\ell)\delta^\ell\ \text{\rm for sufficiently large}\ \ell 
\end{eqnarray}
where $\delta > 0$ and $P,Q$ are eventually positive polynomials. Then 
the limit 
\eqref{E:d} exists and is equal to $\delta$. That is, node $i$ has branching ratio $\delta$.
\end{lemma}

\begin{proof}
Since $P$ and $Q$ are eventually positive,  for  sufficiently large $\ell$ we may take ($1/\ell$)-th powers of \eqref{e:lower_bound} and \eqref{e:upper_bound}:
\[
(Q(\ell))^{1/\ell}  \delta \leq (a_i(\ell))^{1/\ell} \leq (P(\ell))^{1/\ell}  \delta.
\]
If $\ell \to \infty$, then $(Q(\ell))^{1/\ell}$ and $(P(\ell))^{1/\ell}$ both tend to $1$, so
$(a_i(\ell))^{1/\ell}$ tends to $\delta$.
\end{proof}

From now on, 
when considering complex eigenvalues, we
work in $\C^n$ when convenient and take real parts to
reduce to $\R^n$ when necessary.
Recall that the {\em generalised eigenspace} $E_\lambda \subseteq \C^n$ of
an eigenvalue $\lambda$ of an $n \times n$  matrix $M$ is
\[
E_\lambda = \{ \vc v \in \C^n : \vc v (M-\lambda I)^n = 0 \}
\]

\begin{lemma}
\label{L:main_formula}
Let $B(i)$ be the adjacency matrix of the upstream subnetwork $\UU(i)$
of vertex $i$, let $m$ be the number of vertices of $\UU(i)$ and 
$\lambda \neq 0$ run through the eigenvalues of $B(i)$. 
Then there exist polynomials $P_{i\lambda}(\ell)$ such that, for any $\ell \geq m$,
\begin{equation}
\label{E:main_formula}
a_i(\ell) = (\vc u B(i)^\ell)_i = \sum_{\lambda} P_{i\lambda}(\ell) \lambda^\ell.
\end{equation}
\end{lemma}

\begin{proof}
By Proposition \ref{P:upstream_properties} (3) and (7),
we can replace $\GG$ by $\UU(i)$ without changing the value of $a_i(\ell)$;
recall that $\UU(i)$ is the smallest subnetwork containing all walks terminating at $i$,
and we see later that it gives the correct value for the branching ratio.
By property (7) we can then restrict/project $A$ onto the adjacency
matrix $B(i)$ of $\UU(i)$. We then work with $B(i)$, but to simplify
notation we call it $B$, keeping in mind that $B$ depends on $i$.

Let  $E_\lambda$ be the generalised eigenspace of an eigenvalue $\lambda$ of 
$B$, and let $m$ be the number of vertices in $\UU(i)$.
The matrix $B$ has $m$ linearly independent generalised eigenvectors~\cite{B70}, so
\[
\C^m = \bigoplus_\lambda E_\lambda
\]
where here and below sums with variable $\lambda$ range over the eigenvalues of $B$. Decompose $\vc u = [1\ 1\ \ldots 1]$ according to this direct sum, obtaining
\[
\vc u = \sum_\lambda \vc u^\lambda.
\]
By definition, $\vc u^\lambda \cdot (B-\lambda I)^n = 0$ for all $\lambda$.
Let $\delta$ denote the Perron eigenvalue of $B$.
By Lemma \ref{L:in_sum}(b), applied to the network $\UU(i)$ with adjacency
matrix $B$, we have
\begin{equation}
\label{E:a_i_sum}
a_i(\ell) = (\vc u\cdot B^\ell)_i = \left( \sum_\lambda \vc u^\lambda\cdot B^\ell \right)_i
= \sum_\lambda \left(\vc u^\lambda\cdot B^\ell \right)_i.
\end{equation}

We claim that for each $\lambda$ there exists a polynomial $P_\lambda(\ell)$ over $\R$
(which clearly must be eventually positive) such that 
\begin{equation}
\label{E:u_i_bound}
(\vc u^\lambda\cdot B^\ell)_i \leq P_\lambda(\ell) {| \lambda |}^\ell.
\end{equation}

To prove this claim, let the dimension of $E_\lambda$ be $m_\lambda$, and  
let $B_\lambda = B|_{E_\lambda}$.
By the Jordan-Chevalley decomposition \cite{H72}, 
\[
B_\lambda =  K_\lambda + L_\lambda
\]
where $K_\lambda$ is semisimple (diagonalisable) and $L_\lambda$ is nilpotent. 
Indeed, $K_\lambda = \lambda I_\lambda$, where $I_\lambda$
is the identity on $E_\lambda$, as we now show: by the Jordan decomposition,
there is a basis-change matrix $R$ such that
\[
RB_\lambda R^{-1} =  \lambda I_\lambda + M_\lambda
\]
where $M_\lambda$ is nilpotent; so
\[
B_\lambda =  R^{-1}(\lambda I_\lambda  + M_\lambda)R =  \lambda I_\lambda + R^{-1}M_\lambda R
\]
where $R^{-1}M_\lambda R$ is nilpotent. Let $N_\lambda = R^{-1}M_\lambda R$, so 
$B_\lambda = \lambda I_\lambda + N_\lambda$. Since $B$ leaves
$E_\lambda$ invariant and $B_\lambda = B|_{E_\lambda}$,
\[
( {\bf u}^\lambda  B^\ell)_i = ( {\bf u}^\lambda  B_\lambda^\ell)_i = 
( {\bf u}^\lambda  (\lambda I_\lambda + N_\lambda)^\ell)_i .
\] 
Since $I$ and $N_\lambda$ commute, recalling that $m_\lambda$ is the dimension of $E_\lambda$ (hence $N_\lambda^{q}=0$ for all $q \geq m_\lambda$), we can expand by the binomial theorem 
for $\lambda\neq 0$ to get
\beqn
( {\bf u}^\lambda  B^\ell)_i &=&  \left( {\bf u}^\lambda  \left(\lambda^\ell I_\lambda
	+\sum_{p=1}^{m_\lambda} \Matrixc{\ell\\p}\lambda^{\ell-p}N_\lambda^p  \right)\right) _i \\
	&=& \lambda^\ell \left({\bf u}^\lambda\right)_i
	+  \sum_{p=1}^{m_\lambda} \Matrixc{\ell\\p}\lambda^{\ell-p}\left({\bf u}^\lambda N_\lambda^p\right)_i \\
	&=& \lambda^\ell \left( \sum_{p=0}^{m_\lambda} c_p \lambda^{-p} \Matrixc{\ell\\p} \right),
\eeqn
where
\[
c_0 = \left({\bf u}^\lambda\right)_i
\]
and, for every $p=1,\dots,m_\lambda$,
\[
c_p = ({\bf u}^\lambda N_\lambda^p)_i\in \C.
\]
Note that for $\lambda=0$ we have $\vc u^0 B^\ell=\vc u^0 N_0^\ell=0$ for $\ell\geq m_0$, since $N_0$ is nilpotent; as $m_0 \leq m$, the $\lambda = 0$ summand of \eqref{E:a_i_sum} vanishes whenever $\ell \geq m$. For $\lambda \neq 0$, let
\[
P_{i \lambda}(\ell) =  \sum_{p=0}^{m_\lambda} c_p \lambda^{-p}\Matrixc{\ell\\p}.
\]
This is a polynomial over $\C$, and \eqref{E:main_formula} holds.
\end{proof}

\begin{lemma}
\label{L:upper_bound}
An upper bound \eqref{e:upper_bound} exists for any $i$, where $\delta$
equals the Perron eigenvalue of $\UU(i)$.
\end{lemma}
\begin{proof}
By \eqref{E:main_formula},
\beqn
a_i(\ell) &\leq& \sum_\lambda | (\vc u^\lambda \cdot B^\ell)_i | 
	\leq \sum_\lambda |P_\lambda(\ell)| {| \lambda |}^\ell \\
	&\leq& \sum_\lambda |P_\lambda(\ell)| {\delta}^\ell 
	= \left( \sum_\lambda |P_\lambda(\ell)|\right) {\delta}^\ell .
\eeqn

This is not quite the bound we want, because $P_\lambda(\ell)$
is a complex polynomial. However,
for large enough $\ell$ we have $|P_\lambda(\ell)| < \ell^r$ where $r$ is greater than the maximum degree of the $P_\lambda(\ell)$.
So $a_i(\ell) \leq  \ell^r \delta^\ell$ as required.
\end{proof}

\begin{remark}\em
The same proof applies if we work with the adjacency matrix $A$ of $\GG$,
but this gives too large an upper bound if the spectral radius of $\GG$
is greater than that of $\UU(i)$. This is why we work with the adjacency matrix $B$
of $\UU(i)$.
\end{remark}

\begin{remark}\em
For a tighter bound on the degree of the polynomial factor
we can
decompose $P_\lambda(\ell)$ into real and imaginary parts:
 $P_\lambda(\ell) = R_\lambda(\ell) + \ii S_\lambda(\ell)$.
 These have the same degree as $P_\lambda(\ell)$.
 Now,
 \[
 |P_\lambda(\ell)| \leq |R_\lambda(\ell)| + |S_\lambda(\ell)|
 \]
Every real polynomial is zero, eventually positive, or eventually negative.
Therefore 
$|R_\lambda(\ell)| = \pm R_\lambda(\ell)$ and
$|S_\lambda(\ell)| = \pm S_\lambda(\ell)$ for sufficiently large $\ell$ and for suitable choices of signs.
With these choices,
$T_\lambda(\ell) = \pm R_\lambda(\ell) \pm S_\lambda(\ell)$ is either zero or 
eventually positive, has the same degree as $P_\lambda(\ell)$,
and $ |P_\lambda(\ell)| \leq T_\lambda(\ell)$. If $\UU(i)$ is acyclic, $a_i(\ell)=0$ for large $\ell$,
and we can take $P(\ell) = 1$. If, instead, $\UU(i)$ is not acyclic,
$a_i(\ell)>0$ for all $\ell$, so $T_\lambda(\ell)$ cannot be zero: therefore it must be eventually positive.
\end{remark}

It is convenient to have a short term to express condition \eqref{e:lower_bound}:

\begin{definition}\em
\label{D:delta-bounded}
Node $i$ is {\em$\delta$-bounded} if \eqref{e:lower_bound} holds.
\end{definition}

Condition \eqref{e:lower_bound} imposes only a lower bound, so we may want to say $\delta$-lower-bounded, instead, but we drop the `lower' part for the sake of readability.

\begin{proposition}
\label{P:delta_bound}
If $i$ is $\delta$-bounded where $\delta$ is the Perron eigenvalue $\xi(i)$,
then $i$ has branching ratio $\xi(i)$. 
\end{proposition}
\begin{proof}
Lemma \ref{L:upper_bound} gives an upper bound  
\eqref{e:upper_bound}, where $\delta=\xi(i)$, the Perron eigenvalue of $\UU(i)$. Since the vertex 
$i$ is $\delta$-bounded where $\delta = \xi(i)$, the corresponding
lower bound holds. Therefore $i$ has branching ratio $\xi(i)$
by Lemma \ref{L:bounds}.
\end{proof}

Thus the proof of the general theorem follows provided we can
establish a lower bound \eqref{e:lower_bound} with $\delta = \xi(i)$. 
For this reason, the rest of the paper focuses exclusively on lower bounds.

\subsection{Expression for $a_i(\ell)$}

The obvious way to approach this is to split off dominant
terms in \eqref{E:a_i_sum}:
\begin{equation}
\label{E:split_dominant}
a_i(\ell) = \left( \sum_{|\lambda|=\xi} \vc u^\lambda\cdot A^\ell \right)_i + 
	\left( \sum_{|\lambda|<\xi} \vc u^\lambda\cdot A^\ell \right)_i = S_1+S_2,\ {\rm say.}
\end{equation}
As in \eqref{E:u_i_bound}, we expect $S_1$ to be of order $P_1(\ell)\xi^\ell$ for some polynomial $P_1$,
while $S_2$ is of order $P_2(\ell)\sigma^\ell$ for $\sigma<\xi$
and some polynomial $P_2$.
The dominant term is then $S_1$, {\em unless} $S_1 = 0$. This may
seem unlikely, but the entire problem depends upon proving that it cannot happen.

It is difficult to control the sum $S_1$ (we shall discuss this in more detail in Section \ref{S:asymp_h}),
so we proceed
 indirectly by proving that node $i$ is $\xi(i)$-bounded, where $\xi(i)$ is the Perron eigenvalue of $\UU(i)$. This follows from
various standard properties related to the Perron--Frobenius Theorem,
together with simple graph-theoretic observations.

\subsection{Upstream Principle}

Intuitively, we expect the
branching ratio of a given node $i$ to be dominated by  that
of the largest branching ratio of all nodes in $\UU(i)$.
The following easy result 
formalises this intuition for the crucial case of lower bounds.

\begin{proposition}[\bf Upstream Principle]
\label{P:upstream_principle}
Suppose that $j $ is upstream from $i$, and $j$ is $\delta$-bounded.
Then $i$ is $\delta$-bounded.
\end{proposition}
\begin{proof}
There is a walk from $j$ to $i$; let its length be $m$.
Then
\[
a_i(m+\ell) \geq a_j(\ell)
\] 
since any walk of length $\ell$ terminating at $j$ can be concatenated
with the walk from $j$ to $i$ to give a walk of length $m+\ell$ terminating at $i$,
and distinct walks terminating at $j$ yield distinct walks terminating at $i$.
Since $j$ is $\delta$-bounded, there exists $Q(\ell)$, eventually positive,
such that
\[
a_i(m+\ell) \geq a_j(\ell) \geq Q(\ell)\delta^\ell\quad \mbox{for sufficiently large}\ \ell.
\]
Define the polynomial
\[
    \overline Q(x) = \delta^{-m} Q(x-m)
\]
which is itself eventually positive. Then
\[
    \overline Q(m+\ell)=\delta^{-m} Q(\ell)
\]
so
\[
    Q(\ell)\delta^\ell=\delta^{-m}Q(\ell)\delta^{m+\ell}= \overline Q(m+\ell) \delta^{m+\ell}
\]
Therefore
\[
a_i(m+\ell) \geq \overline Q(m+\ell) \delta^{m+\ell}\quad \mbox{for sufficiently large}\ m+\ell,
\]
so $i$ is $\delta$-bounded.
\end{proof}

\section{Perron--Frobenius Theorem for Non-Negative Matrices}
\label{S:PFT}

Next we recall the Perron--Frobenius Theorem, the main technical result
that we need. Everything in this section
is standard. The theorem was originally proved
by Perron \cite{P07} for positive matrices, and later generalised by
Frobenius \cite{F12} to irreducible non-negative matrices (Definition \ref{D:irreducible}). 
To state (the relevant parts of) Frobenius's generalisation, we
need the notions of reducible and irreducible matrices (see Definition \ref{D:irreducible}), and also the following:

\begin{definition}\em
\label{D:period}
Let $M$ be a non-negative square matrix of size $n$, and let $1 \leq i \leq n$. The {\em period of index $i$} of $M$ is the gcd
of all $m>0$ such that $(M^m)_{ii} > 0$. 
If $M$ is irreducible, this is independent of $i$: it is then called the
 {\em period} of $M$ and denoted by $h(M)$ (or just $h$, if $M$ is clear from the context).

 \end{definition}
 
 It is known that if $M$ is irreducible then $h$ is the gcd of all simple cycle lengths in the graph with adjacency matrix $M$ \cite[Section 1.3]{K98}.

We now state the Perron--Frobenius Theorem in a form that
includes only properties that are required in this paper. 

\begin{theorem}[\bf Perron--Frobenius]
\label{T:Perron-Frobenius}
Let $M$ be an irreducible non-negative matrix with period $h$.
Then:

{\rm (a)} $M$ has a (left) eigenvalue $\xi$ (the Perron eigenvalue) 
such that $\xi$ is maximal
among all $|\lambda|$ as $\lambda$ ranges over all eigenvalues. The
number $\xi$ is also called the {\em spectral radius} or {\em Perron eigenvalue} of $M$.

{\rm (b)}  The eigenvalue $\xi$ is real, simple, and positive.
However, other eigenvalues may also have absolute value $\xi$, see {\rm (h)}.

{\rm (c)}  If $\vc v$ is a (left) eigenvector of $M$ for eigenvalue $\xi$ then
all components of $\vc v$ have the same sign, so $\vc v$ can be chosen 
so that all components are positive (the \emph{Perron eigenvector}).

{\rm (d)}  There is also a right eigenvector $\vc w$ of $M$ for eigenvalue $\xi$,
and this is real, all of its components have the same sign so it can be chosen to be positive.

{\rm (e)} The vectors $\vc v, \vc w$ can be normalised, so that $\vc w ^{\rm T} \vc v = 1$.

{\rm (f)} The matrix $\vc v \vc w ^{\rm T}$ is positive.

{\rm (g)} The following property (the left-hand side of which is called the `Ces\`aro average') holds:
\begin{equation}
\label{E:cesaro}
\lim_{k \to \infty} \frac{1}{k}\sum_{\ell=0}^k \xi^{-\ell}M^\ell = \vc v \vc w ^{\rm T}
\end{equation}

{\rm (h)} $M$ has exactly $h$ eigenvalues with absolute value $\xi$.
Each of them is simple and has the form $\zeta^m \xi$, where $\zeta=\ee^{2\pi\ii/h}$ is a primitive $h$-th
root of unity.
\qed\end{theorem}
\begin{proof}
See \cite[Chapter XIII]{G00}, \cite[Section 15.3]{LT85}, and \cite[Chapter 7]{M00}.
\end{proof}

\ignore{
\begin{remark}\em\footnote{\RED{Remark \ref{R:P-values} could be omitted.}}
\label{R:P-values}
It is natural to ask: What values can the branching ratio $\delta(i)$ take,
for suitable networks $\GG$?
Equivalently: What are the possible values of the Perron eigenvalue $\xi$ of
a nonnegative integer matrix?
This is  the `Perron--Frobenius inverse eigenvalue problem'.
There are some obvious necessary conditions. 

{\rm (a)} Either $\xi = 0$ or $\xi \geq 1$.

{\rm (b)} $\xi$ is an algebraic integer.

{\rm (c)} If $\sigma$ is a Galois conjugate (root of the minimum polynomial) of $\xi$ then $|\sigma| \leq \xi$.

\noindent
Lind \cite[Theorem 1]{L84} proves that these 
necessary conditions are also sufficient, and gives an
algorithm to construct a suitable matrix from the minimum polynomial
of $\xi$. The size of the matrix must sometimes be larger 
than the degree of this polynomial. 
The proof is technical, and is based on ergodic theory. 
\end{remark}
}

\section{Branching Ratios in Strongly Connected Networks}
\label{S:BRSCNP}

Suppose that $\GG$ is strongly connected, with any period $h$.
Then $A$ is irreducible. The main result is:

\begin{theorem}
\label{T:BRexists_any_h}
Let $\GG$ be a strongly connected network 
and let $i$ be any node.  Then 

{\rm (a)} Node $i$ has a well-defined branching ratio $\delta(i)$.
That is, \eqref{E:d} holds for $\delta=\delta(i)$.

{\rm (b)} For all $i$ we have $\delta(i) = \xi(i)$, where $\xi(i)$ is the Perron eigenvalue
of the adjacency matrix of $\GG$.

{\rm (c)} All nodes of $\GG$ have the same branching ratio.
\end{theorem}

The proof is based on \eqref{E:main_formula} and the Ces\`aro average
\eqref{E:cesaro}. (There is a simpler proof when $h=1$ but we omit this.)
First, we prove a simple lemma.
\begin{lemma}
\label{L:Fd(z)}
Let $z \in \C$ and $d \in \N$. Define the power series
\[
F_d(z) = \sum_{\ell=0}^\infty \ell^d z^\ell
\]
Then 

\noindent
{\rm (a)} The radius of convergence of the series is $1$.

\noindent
{\rm (b)} 
There exists a polynomial $G_d(z)$ such that
\begin{equation}
\label{E:Fd(z)}
F_d(z) = \frac{G_d(z)}{(1-z)^{d+1}} \qquad (|z|<1)
\end{equation}
\end{lemma}
\begin{proof}
Part (a) follows by the ratio test.

We prove (b) by induction on $d$. When $d=0$,
\[
F_0(z) = \sum_{\ell=0}^\infty z^\ell = \frac{1}{1-z}\qquad (|z|<1)
\]
When $|z|<1$ we may differentiate term-by-term, to obtain
\[
F'_d(z) = \sum_{\ell=0}^\infty \ell^d \ell z^{\ell-1}  = \frac{1}{z}F_{d+1}(z)
\]
so
\[
F_{d+1}(z) = zF'_d(z)
\]
By the inductive hypothesis \eqref{E:Fd(z)} for $d$,
\beqn
F_{d+1}(z) &=& z \frac{\rm d}{{\rm d} z}\left( \frac{G_d(z)}{(1-z)^{d+1}} \right)\\
	&=& z \left( \frac{G'_d(z)}{(1-z)^{d+1}} +G_d(z) \frac{d+1}{(1-z)^{d+2}} \right)\\
	&=&  \frac{G_{d+1}(z)}{(1-z)^{d+2}}
\eeqn
where
\[
G_{d+1}(z) = z(1-z)G'_d(z)+(d+1)G_d(z)
\]
This completes the induction.
\end{proof}

\begin{proof}[Proof of Theorem {\rm \ref{T:BRexists_any_h}}]

Recall that $B=B(i)$ is the adjacency matrix of $\UU(i)$, and let $\lambda$
run through the eigenvalues of $B$.
By Lemma \ref{L:main_formula}, there exist polynomials 
$P_{i\lambda}(\ell)$ such that

\begin{equation}
\label{E:main_formula_dup}
a_i(\ell) = (\vc u B^\ell)_i = \sum_{\lambda} P_{i\lambda}(\ell) \lambda^\ell = S_1+S_2
\end{equation}
where we have split the sum into two parts:
\[
S_1 = \sum_{|\lambda|=\xi} P_{i\lambda}(\ell) \lambda^\ell \qquad\quad
S_2 = \sum_{|\lambda|<\xi} P_{i\lambda}(\ell) \lambda^\ell
\]
and for simplicity we write $\xi = \xi(i)$.

Since $B$ is irreducible, the eigenvalues $\lambda$ with $|\lambda| = \xi$ are
simple by Theorem \ref{T:Perron-Frobenius}(h), so the corresponding polynomials
$P_{i\lambda}$ are constants. Therefore
\[
S_1 = C_\ell\, \xi^\ell \qquad \mbox{where}\quad
C_\ell = \sum_{|\lambda|=\xi} P_{i\lambda}\left(\frac{\lambda}{\xi}\right)^{\!\ell},
\]
and $C_\ell$ depends only on $\ell \bmod h$, because each ratio $\lambda/\xi$
is an $h$-th root of unity.

We claim that $S_1$ does not vanish identically; that is, $C_r \neq 0$ for at
least one residue $r$ modulo $h$. We prove the claim
first, and then derive the theorem.

Let $\vc v$ be a left eigenvector of $B$ for eigenvalue $\xi$, and let
$\vc w$ be a right eigenvector of $B$ for eigenvalue $\xi$, as in Theorem \ref{T:Perron-Frobenius}(c) and (d). Both vectors are positive.

Suppose, for a contradiction, that $S_1 \equiv 0$ for all $\ell$. If so,
\eqref{E:main_formula_dup} implies that
\begin{equation}
\label{E:S2eq}
(\vc u B^\ell)_i=a_i(\ell) = S_2 = \sum_{|\lambda|<\xi} P_{i\lambda}(\ell) \lambda^\ell
\end{equation}
which is obviously of order $o(\xi^\ell)$. We show
that this contradicts the Ces\`aro average formula \eqref{E:cesaro}.

By Theorem \ref{T:Perron-Frobenius}(f) the matrix $ \vc v \vc w^{\rm T}$ is positive.
By Theorem \ref{T:Perron-Frobenius}(g),
\[
\vc v \vc w^{\rm T} = \lim_{k \to \infty} \frac{1}{k}\sum_{\ell =0}^k B^\ell \xi^{-\ell}
\]
Therefore, by \eqref{E:S2eq},
\[
\left(\vc u(\vc v \vc w^{\rm T})\right)_i = \left(\lim_{k \to \infty} \frac{1}{k}\sum_{\ell =0}^k \vc u B^\ell \xi^{-\ell}\right)_i
	=\sum_{|\lambda|<\xi}  \left(\lim_{k \to \infty} \frac{1}{k} \sum_{\ell =0}^k P_{i\lambda}(\ell)\left(\frac{\lambda}{\xi}\right)^{\ell}\right).
\]
We claim that the term inside parentheses in the right-hand expression
vanishes for each $\lambda$.
To see why, fix $\lambda$ and let $z_\lambda = \lambda/\xi$. 
Choose $d \in \N$ such that $P_{i\lambda}(\ell) < \ell^d$ for sufficiently
large $\ell$; this can be done by taking $d$ to be greater than the maximal degree
of the $P_{i\lambda}(\ell)$. Now
\beqn
|(\vc u(\vc v \vc w^{\rm T}) )_i | &\leq& \sum_{|\lambda|<\xi}  \left(\lim_{k \to \infty} \frac{1}{k} \sum_{\ell =0}^k |P_{i\lambda}(\ell)|\left|\frac{\lambda}{\xi}\right|^{\ell}\right)\\
	&\leq & \sum_{|\lambda|<\xi}  \left(\lim_{k \to \infty} \frac{1}{k} \sum_{\ell =0}^k \ell^d |z_\lambda|^\ell\right).
\eeqn
Since the eigenvalues $\lambda$ occurring in this sum
satisfy $|\lambda| < \xi$, we have $|z_\lambda| < 1$, within the radius of convergence of the series \eqref{E:Fd(z)}.
By Lemma \ref{L:Fd(z)}(b) the sum to infinity of this series equals $G_d(z_\lambda)(1-z_\lambda)^{-d-1}$.
Therefore the sum to $k$ terms is as close as we wish to this value,
so in particular there exists $K$ such that when $|z_\lambda| < 1$,
\[
\left|  \sum_{\ell =0}^k \ell^d |z_\lambda|^\ell \right| \leq 2 G_d(z_\lambda)(1-z_\lambda)^{-d-1} \qquad (k \geq K)
\]
Now
\[
\lim_{k \to \infty} \frac{1}{k} \sum_{\ell =0}^k \ell^d |z_\lambda|^\ell = 
	\lim_{k \to \infty} \frac{1}{k}\left(\frac{2 G_d(z_\lambda)}{(1-z_\lambda)^{d+1}}\right) = 0
\]
since $z_\lambda$ is constant for a given $\lambda$. This implies that
\[
(\vc u(\vc v \vc w^{\rm T}) )_i = 0
\]
However, the vector $\vc u$ and the matrix $\vc v \vc w^{\rm T}$ are both positive,
so $\vc u(\vc v \vc w^{\rm T})$ is positive. This is a contradiction,
which proves the claim that $S_1$ does not vanish identically.

Having proved the claim, we deduce the theorem. The sum $S_2$ has
order $o(\xi^\ell)$ since all $\lambda$ occurring in $S_2$
satisfy $|\lambda| < \xi$. Therefore
$a_i(\ell) = C_\ell\,\xi^\ell + o(\xi^\ell)$, so for each residue $r$ modulo $h$,
\[
C_r = \lim_{\substack{\ell \to \infty \\ \ell \equiv r \pmod{h}}} \frac{a_i(\ell)}{\xi^\ell}.
\]
In particular each $C_r$ is real and non-negative, since $a_i(\ell) \geq 0$.
By the claim, $C_{r^*} > 0$ for at least one residue $r^*$.

We now extend this lower bound to every residue class by a monotonicity argument.
Since $\GG$ is strongly connected and not acyclic (otherwise $\xi = 0$ and the
theorem holds trivially), every node of $\GG$ has in-degree at least $1$.
Hence the map that deletes the first edge, from walks of length $\ell+1$
terminating at $i$ to walks of length $\ell$ terminating at $i$, is surjective,
so
\[
a_i(\ell+1) \geq a_i(\ell) \quad \mbox{for all}\ \ell.
\]
Given $\ell$, let $\ell''$ be the largest integer such that $\ell'' \leq \ell$
and $\ell'' \equiv r^* \pmod{h}$, so that $\ell - \ell'' < h$. For sufficiently
large $\ell$,
\[
a_i(\ell) \geq a_i(\ell'') \geq \tfrac{1}{2}\,C_{r^*}\,\xi^{\ell''}
	\geq \left(\tfrac{1}{2}\,C_{r^*}\,c_0\right)\xi^{\ell}
\]
where $c_0 = \min(1,\xi^{-h}) > 0$. Taking $Q(\ell)$ to be the constant
(hence eventually positive) polynomial $\frac{1}{2}C_{r^*}c_0$, we obtain
\begin{equation}
    \label{eqn:qlower}
    a_i(\ell) \geq Q(\ell)\xi^\ell
\end{equation}
for sufficiently large $\ell$. Incidentally, combined with
$a_i(\ell) = C_\ell\,\xi^\ell + o(\xi^\ell)$, this shows that $C_r > 0$ for
\emph{every} residue $r$.

On the other hand, the final step in the proof of
Lemma \ref{L:upper_bound}
implies that there exists an eventually positive polynomial
$\ell^r$, for suitably large $r$, such that
\begin{equation}
\label{E:R_bound}
a_i(\ell) \leq \ell^r\xi^\ell
\end{equation}
for sufficiently large $\ell$. Now Lemma \ref{L:bounds} implies parts (a) and (b)
of the theorem. 
Part (c) follows since we are assuming $B$ is irreducible,
so $\UU(i)$ is strongly connected.
\end{proof}

\section{Branching Ratios in the General Case}
\label{S:BRGC1}

With very little extra effort we can parlay Theorem \ref{T:BRexists_any_h} into
a more general result that applies to any network, without assuming irreducibility.
We first need the following:
\begin{definition}\em
\label{D:upstream_SCC}
An {\em upstream SCC} of node $i$ is a 
strongly connected component (SCC) of $\GG$ that contains a node upstream from $i$.
\end{definition}

\begin{theorem}
\label{T:gencase}
Let $i$ be any node of a network $\GG$. Then:

{\rm (a)} Node $i$ has a well-defined branching ratio $\delta(i) \geq 0$.

{\rm (b)} The value of $\delta(i)$ is the maximum value of the $\delta(j)$
as $j$ runs over all the upstream SCCs of $i$.

{\rm (c)} The value of $\delta(i)$ is equal to
the Perron eigenvalue $\xi(i)$ of the
adjacency matrix $B(i)$ of the upstream subgraph $\UU(i)$,
as in Proposition {\rm \ref{P:upstream_properties}(7)}.
\end{theorem}

Before giving the proof we prove a lemma.

\begin{lemma}
\label{L:upstream_ev}
$ $

{\rm (a)} The SCCs of $\UU(i)$ are the upstream SCCs of $i$.

{\rm (b)} $\UU(i)$ is the subgraph induced by the union of all
upstream SCCs of $i$.

{\rm (c)} The eigenvalues of the adjacency matrix of $\UU(i)$ are
the union (counting multiplicities) of the eigenvalues of all
upstream SCCs of $i$.

{\rm (d)} The Perron eigenvalue $\xi(i)$ is the maximum of
the Perron eigenvalues of the upstream SCCs of $i$.
\end{lemma}
\begin{proof}
Parts (a) and (b) follow immediately since each SCC is strongly connected.

Parts (c) and (d) follow from the block upper triangular form
\eqref{E:block_B} of the adjacency matrix $B(i)$ of $\UU(i)$
(Proposition \ref{P:upstream_properties}(7); see also
\cite[Theorem 4.11]{GS23}), whose diagonal blocks are the adjacency
matrices of the SCCs of $\UU(i)$.
\end{proof}

\begin{proof}[Proof of Theorem {\rm \ref{T:gencase}}]

The Perron eigenvalue $\xi(i)$ of $\UU(i)$ is equal to the Perron
eigenvalue of some SCC of $\UU(i)$. Let $S$ denote this SCC.
$S$ is obviously strongly connected, so every node of $S$ is $\xi(i)$-bounded by \eqref{eqn:qlower}.
 By The Upstream Principle,
node $i$ is $\xi(i)$-bounded, as well. Therefore the branching ratio $\delta(i)$ of
node $i$ as the limit \eqref{E:d} is defined, proving (a).
Also $\delta(i) = \xi(i)$, proving (c). Finally,
(b) follows from Lemma \ref{L:upstream_ev}.
\end{proof}

\begin{corollary}
\label{C:rho-bounded}
For any network $\GG$ and node $i$, the branching ratio $\delta(i) = \delta$
if and only if there exist eventually positive polynomials $P(\ell)$
and $Q(\ell)$
such that $Q(\ell)\delta^\ell \leq a_i(\ell) \leq P(\ell)\delta^\ell$
for sufficiently large $\ell$.
\qed\end{corollary}

We also deduce a well-known result:
\begin{corollary}
\label{C:rho-values}
The Perron eigenvalue of any non-negative integer matrix
is either zero or greater than or equal to $1$.
\qed\end{corollary}

\begin{proof}
If the branching ratio is nonzero then $a_i(\ell) \geq 1$
for all $\ell$. So $a_i(\ell)^{1/\ell} \geq 1$.
\end{proof}

\section{Asymptotics}
\label{S:A}

We have now established the basic properties of the branching ratio,
for all networks. We now proceed in a new direction: finding
more precise asymptotic expressions for $a_i(\ell)$ under different
hypotheses. Again, the examples in Section \ref{S:ME} motivate the results.

Our main results are:

(a) If $\UU(i)$ is strongly connected of period $h$, and is not acyclic,
then for every $0 \leq r \leq h-1$, there exist constants $C_{r,i}>0$, depending on $r$ and $i$, 
such that $a_i(\ell) \sim C_{r,i}\xi^\ell$ when $r = \ell \pmod{h}$.
This follows by applying Theorem \ref{T:periodic_asymp} to $\UU(i)$.

(b) As a corollary, if $\UU(i)$ is strongly connected of period 1, and is not acyclic,
then there exists $C_i>0$, depending on $i$, such that $a_i(\ell) \sim C_i\xi^\ell$.
This follows by applying Corollary \ref{C:asymp1} to $\UU(i)$.

Example \ref{no_ratio} shows that when the period
is greater than $1$ the constants $C_{r,i}$
depend on both $r$ and $i$. For (b), the only
value of $r$ is $r=0$, so the constant depends only on $i$.

(c) Define $g$ to be the
least  common multiple of the periods of those adjacency matrices
of the SCCs of $\UU(i)$ that have Perron eigenvalue $\xi$.
Then there exist eventually positive real polynomials 
$R_0(\ell),R_1(\ell),\ldots, R_{g-1}(\ell)$ such that
$a_i(\ell) \sim R_s(\ell)\xi^\ell$  where $s = \ell \pmod{g}$.
See Theorem \ref{T:asymp_general}.

\subsection{Asymptotics for Irreducible Adjacency Matrices}
\label{S:asymp_h}

We begin with the irreducible case.
(The results in this section also hold when the adjacency matrix of $\UU(i)$ is irreducible, so that we can apply them to each $\UU(i)$ separately. We do not state these versions of the results.)

\begin{theorem}
\label{T:periodic_asymp}
Let the adjacency matrix $A$ of $\GG$ be irreducible with period $h$ and Perron eigenvalue $\xi$.
Then there exist constants $C_r > 0$,
for $0 \leq r \leq h-1$, such that when $\ell \pmod{h}=r$ then
\[
a_i(\ell) \sim C_r \xi^\ell
\]
These constants depend on $i$.
\end{theorem}
\begin{proof}
Since $A$ is irreducible, $\GG$ is strongly connected, so $\UU(i)=\GG$ for all $i$.
If $\UU(i)$
is acyclic, $\xi = 0$ and any constants $C_r > 0$ can be used.
Therefore we can assume that $\UU(i)$ is not acyclic.

By \eqref{E:split_dominant}, 
\[
a_i(\ell) =S_1(\ell)+S_2(\ell)
\]
where
\[
S_1(\ell) =  \sum_{|\lambda|=\xi} (\vc u^\lambda.B^\ell)_i,\qquad 
S_2(\ell) = \sum_{|\lambda|<\xi} (\vc u^\lambda.B^\ell)_i
\]
For simplicity, write $\xi(i)=\xi$.
By \eqref{eqn:qlower} there is an eventually positive polynomial $Q$ such that
$a_i(\ell) \geq Q(\ell)\xi^\ell$ for sufficiently large $\ell$. Since
$S_2(\ell) = o(\xi^\ell)$, the sum $S_1(\ell) = a_i(\ell) - S_2(\ell)$ does not
vanish for sufficiently large $\ell$. Now $a_i(\ell) \sim S_1(\ell)$, and
we seek an expression for $S_1(\ell)$.

By Theorem \ref{T:Perron-Frobenius}(h), the eigenvalues $\lambda$
with $|\lambda|= \xi$ are
\[
\lambda = \xi, \zeta\xi,  \zeta^2\xi, \ldots,  \zeta^{h-1}\xi, 
\]
and these eigenvalues are simple.
Let $\lambda_k = \zeta^k\xi$ for $0 \leq k \leq h-1$, with eigenvector
$\vc u^{\lambda_k}$. Now 
\[
\vc  u^{\lambda_k} B = \lambda_k \vc  u^{\lambda_k} = \zeta^k\xi\, \vc  u^{\lambda_k} \quad {\rm so} \quad
	\vc  u^{\lambda_k}B^\ell = \lambda_k^\ell  \vc u^{\lambda_k} = \zeta^{k\ell}\xi^{\ell}\,  \vc u^{\lambda_k}
\]
Therefore
\begin{equation}
\label{E:asymp_h}
a_i(\ell) \sim S_1(\ell) = \left(\sum_{k=0}^{h-1} \zeta^{k\ell} \vc  u^{\lambda_k}_i \right)\xi^{\ell} = C_\ell \xi^\ell
\end{equation}
where
\[
C_\ell = \sum_{k=0}^{h-1} \zeta^{k\ell} \vc  u^{\lambda_k}_i \in \C
\]
However, the non-real eigenvalues $\lambda_k$ and their eigenvectors 
$\vc  u^{\lambda_k}$ occur in conjugate pairs, so actually $C_\ell$ is real.

Since $\zeta^h = 1$ we have
\[
C_{\ell+h} = \sum_{k=0}^{h-1} \zeta^{k(\ell+h)} \vc u^{\lambda_k}_i
	=  \sum_{k=0}^{h-1} \zeta^{k\ell} \vc u^{\lambda_k}_i 
	= C_\ell
\]
so $C_\ell$ depends only on $\ell \pmod{h}$. That is, if
$\ell = qh+r$, where $0 \leq r \leq h-1$, we can define $C_\ell = C_r$,
and \eqref{E:asymp_h} remains valid. That is,
\begin{equation}
\label{E:asymp_h2}
a_i(\ell) \sim C_r\xi^\ell\ \mbox{where}\ r = \ell \pmod{h}
\end{equation}
Finally, the $C_r$ must be positive by \eqref{E:asymp_h2}.
\end{proof}

Example \ref{no_ratio} shows that when $h>1$, precise asymptotics may need
$h$ different constant factors. In that example 
these constants are in geometric progression, but 
Example \ref{E:polycycle} shows that this need not happen.

The period-1 case is worth being stated explicitly:
\begin{corollary}
\label{C:asymp1}
Let $\GG$ be strongly connected of period $1$, and let $\xi$ be its Perron eigenvalue.
Then there exists a constant $C > 0$ such that $a_i(\ell) \sim C\xi^\ell$ as $\ell \to \infty$.
\qed\end{corollary}
The constant $C$ may depend on $i$, see for example \eqref{E:fibo_asymp}.
When $\UU(i)$ is acyclic, any $C$ can be used since $\xi=0$.

\subsection{Asymptotics for the General Case}
\label{S:AGC}

The results of Theorem \ref{T:periodic_asymp} 
do not carry over to the case when $\UU(i)$ is not strongly connected. In general,
the constants $C_r$ must be replaced by eventually positive polynomials in $\ell$.
Consider, for instance, Example \ref{abff}: there
\[
a_1(\ell)=\alpha^\ell+\ell \alpha^{\ell-1}\beta \qquad a_2(\ell) = \alpha^\ell 
\]
and $\xi = \alpha$. Here $a_1(\ell)$ is not asymptotic to $C\alpha^\ell$ for
a constant $C$, because
\[
\frac{a_1(\ell)}{C\alpha^\ell} = \frac{\alpha^\ell+\ell \alpha^{\ell-1}\beta}{C\alpha^\ell} = \frac{1}{C}+ \ell\frac{\beta}{C\alpha}
\]
which does not tend to 1 as $\ell \to \infty$ (unless $\beta=0$, which was explicitly ruled out).

We now derive the general form of the asymptotics of
$a_i(\ell)$ when $\UU(i)$ is not acyclic. Again, if $\UU(i)$ is acyclic
then $a_i(\ell) = 0$ for large $\ell$.

\subsubsection{Preliminary Remarks}

Before embarking on a precise analysis, we comment on the
kind of theorem that we should expect.

The upstream subnetwork $\UU(i)$ is induced by all upstream
SCCs.  The Perron eigenvalue $\xi(i)$ is the maximum of the
Perron eigenvalues of those SCCs. The block upper triangular form of the
adjacency matrix suggests that only those SCCs whose Perron eigenvalue
is equal to this maximum affect the dominant asymptotics. 

Any SCC has an irreducible adjacency matrix, so the branching ratios
for these SCCs all equal $\xi(i)$, which we now write as just $\xi$.  

We know that for the irreducible case, $a_i(\ell)$ is asymptotic to
$P_k(\ell)\xi^\ell$ for polynomials $P_k$, where $k \equiv \ell\pmod{h}$
and $h$ is the period.

Consider the simple case of a network in which node $i$ forms a single SCC, with inputs
from two other SCCs, say $\HH_1$ and $\HH_2$, whose tails are $j_1, j_2$. 
See Fig. \ref{F:SCCs}.

\begin{figure}[h!]
\centerline{%
\includegraphics[width=0.4\textwidth]{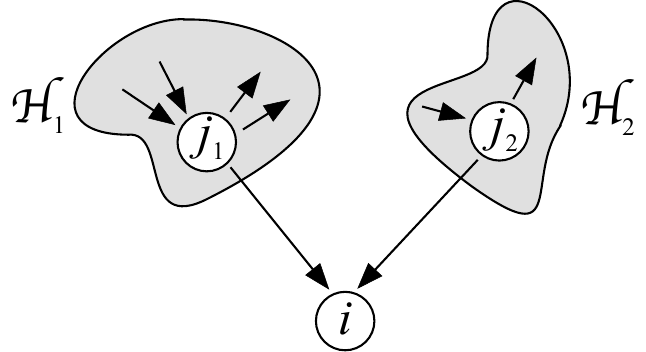}}
\caption{Node with two upstream SCCs.}
\label{F:SCCs}
\end{figure}

Suppose that $\HH_1$
has period $h_1$ and $\HH_2$ has period $h_2$. Any walk
terminating at $i$ lies entirely within one of $\HH_1,\HH_2$, except
for the final edge to vertex $i$. Therefore, for $\ell > 0$,
\[
a_i(\ell) = a^1_{j_1}(\ell-1)+a^2_{j_2}(\ell-1)
\]
where, for $m=1,2$, the quantity $a^m_{j_m}(\ell)$ is the 
number of walks of length $\ell$ in $\HH_m$ terminating at $j_m$.
Now, 
\[
a^m_{j_m}(\ell) \sim P^m_k(\ell)\xi^\ell
\]
where the polynomials $P^m_k$ depend on $k \equiv \ell\pmod{h_m}$.
Therefore, by a shift in the polynomial argument and absorbing a factor $\xi^{-1}$ in their coefficients, there exist
polynomials $Q_k^i$ such that
\[
a_i(\ell) \sim (Q^1_{k_1}(\ell) + Q^2_{k_2}(\ell))\xi^\ell
\]
where $k_1$ depends on $\ell\pmod{h_1}$ and $k_2$ depends on $\ell\pmod{h_2}$.
The sum of these two polynomials therefore depends on  $\ell\pmod{g}$,
where $g= \mathrm{lcm}\,(h_1,h_2)$.

This situation is perhaps the simplest example where two different upstream SCCs
input to a given node. It suggests that in general we should have
asymptotic estimates of the form
\[
a_i(\ell) \sim R_k(\ell)\xi^\ell
\]
where the polynomials $R_k$ depend on $k = \ell\pmod{g}$, where $g$
is the {\em least common multiple} of the periods of the SCCs with 
Perron eigenvalue $\xi$.

\subsubsection{Statement and Proof for General Case}

We now prove that this heuristic gives the correct result.

\begin{theorem}
\label{T:asymp_general}
Let $\GG$ be any network, let $i$ be a fixed but arbitrary node,
and let the Perron eigenvalue of $\UU(i)$ be $\xi$. Define $g$ to be the
least  common multiple of the periods of those adjacency matrices
of the SCCs of $\UU(i)$ that have Perron eigenvalue $\xi$.
Then there exist eventually positive real polynomials 
$R_0(\ell),R_1(\ell),\ldots, R_{g-1}(\ell)$ such that
\[
a_i(\ell) \sim R_s(\ell)\xi^\ell\quad {\rm where}\ s = \ell \pmod{g}.
\]
These polynomials depend on $i$ in general; however, $\xi$ and $g$ are the same for all nodes in the same SCC of $\GG$.

\end{theorem}
\begin{proof}
Let $\xi$ be the Perron eigenvalue of $\UU(i)$. Let the SCCs of $\GG$ that are contained in
$\UU(i)$ be $\HH_1, \ldots, \HH_t$. By Proposition \ref{P:same_P_ev}, $\xi(j)=\xi(k)$ whenever
$j$ and $k$ belong to the same SCC $\HH_s$. Define $\xi(\HH_s) = \xi(j)$
for any such $j$. Then
\[
\xi = \max\{\xi(\HH_s): 1 \leq s \leq t\}
\]
Define $T \subseteq \{1, \ldots,t\}$ by
\[
T = \{s: \xi(\HH_s) = \xi\}
\]
This is the set of (indices of) SCCs of $\UU(i)$ with Perron eigenvalue equal to $\xi$.

We show that it is these SCCs that determine the asymptotics of $a_i(\ell)$.

Number the $\HH_s$ so that $\HH_{s_1}$ being upstream from $\HH_{s_2}$
implies that $s_2 \geq s_1$. That is, the only edges with source in
 $\HH_{s_1}$ have targets in $\HH_{s_2}$ where $s_2 \geq s_1$. 
Let the period of $\HH_s$ be $h_s$.
 
If $B$ is the adjacency matrix of $\UU(i)$ then, ordering nodes
in successive blocks according to their SCCs $\HH_s$
 as in Proposition \ref{P:upstream_properties}(6),
$B$ is block upper triangular as in \eqref{E:block_B},
where the block $B_s$ is the adjacency matrix of $\HH_s$.

By Lemma \ref{L:upstream_ev} the eigenvalues of $B$ are the union of those of the $B_s$. Moreover, the block
$B_s$ corresponds to the subspace $V_s$ spanned by node basis vectors $\vc e_i$
such that node $i$ belongs to $\HH_s$. Let $d_s = \dim V_s$.
By Lemma \ref{L:main_formula},
\begin{equation}
\label{E:gen_poly_sum}
a_i(\ell) = (\vc u.B^\ell)_i = \sum_\lambda (\vc u^\lambda . B^\ell)_i 
	= \sum_\lambda P_{i\lambda}(\ell)\lambda^\ell
\end{equation}
for suitable (complex) polynomials $P_{i\lambda}(\ell)$.

By Proposition \ref{P:upstream_principle} and Theorem \ref{T:periodic_asymp}, 
$a_i(\ell) \geq C\xi^\ell$ for some $C > 0$. (Choose $C$ to be the minimum of the constants $C_r$.)
This implies that the dominant terms in \eqref{E:gen_poly_sum}, which determine the
asymptotics, are those for which $|\lambda| = \xi$. Moreover, those
terms have nonzero sum, since otherwise $a_i(\ell) = o(\xi^\ell)$. Therefore
\[
a_i(\ell) \sim \sum_{|\lambda|=\xi} P_{i\lambda}(\ell)\lambda^\ell.
\]
The only eigenvalues $\lambda$ with $|\lambda|=\xi$ are those that
occur for the block matrices $B_s$ with $s \in T$. By Theorem \ref{T:Perron-Frobenius}(h) these are
\[
\zeta_s^m\xi \qquad (0 \leq m < h_s)
\]
where $\zeta_s = \ee^{2\pi\ii/h_s}$. Therefore
\[
a_i(\ell) \sim \sum_{|\lambda|=\xi} P_{i\lambda}(\ell)\lambda^\ell
	= \sum_{s\in T} \sum_{m=0}^{h_s-1} P_{i,\zeta_s^m\xi}(\ell)(\zeta_s^m\xi)^\ell
\]
Now
\[
 P_{i,\zeta_s^m\xi}(\ell)(\zeta_s^m\xi)^\ell = \zeta_s^{m\ell} P_{i,\zeta_s^m\xi}(\ell)\xi^\ell
\]
Let
\[
Q^s_{im}(\ell) =   \zeta_s^{m\ell} P_{i,\zeta_s^m\xi}(\ell)
\]
Then
\[
a_i(\ell) \sim \left( \sum_{s\in T} \sum_{m=0}^{h_s-1} Q^s_{im}(\ell) \right)\xi^\ell.
\]
Each function $Q^s_{im}(\ell)$ depends only on $\ell \pmod{h_s}$.
Therefore the sum over all $s \in T$ depends only on $\ell \pmod{g}$,
where $g ={\rm lcm}\, (h_s)$. That is, there are polynomials
$R_0(\ell), \ldots, R_{g-1}(\ell)$, such that
\begin{equation}
\label{E:Rsum}
a_i(\ell) \sim R_k(\ell)\xi^\ell \qquad \mbox{where}\ k \equiv \ell \pmod{g}
\end{equation}
We can assume these polynomials are real by taking real parts.
Since $a_i(\ell)>0$, equation \eqref{E:Rsum} implies that the $R_k$ are eventually positive.
\end{proof}

Example \ref{abff} shows that in general the polynomials $R_s(\ell)$ cannot be
taken to be constants, and Example \ref{no_ratio} shows that they can differ
for different residues $s$.

\section*{Acknowledgements}

This work was supported in part by project SERICS (PE00000014) under the NRRP
MUR program funded by the EU - NGEU. Views and opinions expressed are however
those of the authors only and do not necessarily reflect those of the European
Union or the Italian MUR. Neither the European Union nor the Italian MUR can be
held responsible for them.

We thank the reviewers for comments that led to several improvements to this paper.

\section*{Competing interests}

The authors have no competing interests to declare that are relevant to the content of this article.

\end{document}